\documentclass[a4paper, 10pt,reqno]{amsart}  

\usepackage{amssymb,amsmath,amsthm, mathtools, dsfont}
\usepackage{stmaryrd}
\usepackage{enumitem,color}
\usepackage{hyperref}
\hypersetup{hidelinks}
\usepackage[nameinlink]{cleveref}
\usepackage{nicefrac}
\usepackage{microtype}

\Crefname{theorem}{Theorem}{Theorems}
\Crefname{definition}{Definition}{Definitions}
\Crefname{lemma}{Lemma}{Lemma}
\Crefname{remark}{Remark}{Remarks}
\Crefname{conjecture}{Conjecture}{Conjectures}
\Crefname{proposition}{Proposition}{Proposition}

\newtheorem{theorem}{Theorem}
\newtheorem{definition}[theorem]{Definition}
\newtheorem{lemma}[theorem]{Lemma}
\newtheorem{corollary}[theorem]{Corollary}
\newtheorem{remark}[theorem]{Remark}

\newtheorem{assumption}{Assumption}
\newtheorem{example}[theorem]{Example}
\newtheorem{proposition}[theorem]{Proposition}
\newtheorem*{mainresult}{Main result}

\newcommand{\e}{\mathrm{e}}
\newcommand{\C}{\mathbb{C}}
\newcommand{\N}{\mathbb{N}}
\newcommand{\R}{\mathbb{R}}

\newcommand{\iu}{\mathrm{i}}
\renewcommand{\d}{\, \mathrm{d}}
\newcommand{\Lb}{\mathcal{L}}
\newcommand{\Lp}{\mathrm{L}}

\newcommand{\E}{\mathrm{E}}
\newcommand{\W}{\mathrm{W}}
\newcommand{\set}[2]{\left\{ #1 \: \middle \vert \: #2 \right\}}

\DeclareMathOperator{\real}{\mathrm{Re}}
\DeclareMathOperator{\imag}{\mathrm{Im}}
\DeclareMathOperator{\ran}{\mathrm ran}
\DeclareMathOperator{\sgn}{\mathrm{sgn}}

\title{Characterization of Orlicz admissibility}

\allowdisplaybreaks

\numberwithin{theorem}{section}

\begin{document}

\thanks{This work has been supported by the German Research Foundation (DFG) via the joint grant JA 735/18-1 / SCHW 2022/2-1.\\
The first author would like to thank Bernhard Haak, University of Bordeaux, for valuable discussions on square function estimates during his six-month stay at the University of Bordeaux.}


%
%
%
%
%
%
 
 \author[R.~Hosfeld]{Ren\'e Hosfeld}
\address[RH]{University of Wuppertal, School of Mathematics and Natural Sciences, IMACM, Gau\ss \-str.\ 20, 42119 Wuppertal, Germany and Department of Mathematics, University of Hamburg, Bundesstra\ss e 55, 20146 Hamburg, Germany}
\email{hosfeld@uni-wuppertal.de}
\author[B.~Jacob]{Birgit Jacob}
\address[BJ]{University of Wuppertal, School of Mathematics and Natural Sciences, IMACM, Gau\ss \-str.\ 20, 42119 Wuppertal, Germany}
\email{bjacob@uni-wuppertal.de}
\author[F.L.~Schwenninger]{Felix L.~Schwenninger}
\address[FLS]{ Department of Applied Mathematics, University of Twente, P.O. Box 217, 7500 AE Enschede, The Netherlands and Department of Mathematics, University of Hamburg, Bundesstra\ss e 55, 20146 Hamburg, Germany}
\email{f.l.schwenninger@utwente.nl}


\begin{abstract}
In this note we extend two characterizations of admissible operators with respect to $\Lp^p$ to more general Orlicz spaces. The equivalent conditions are
given by the property that an associated operator generates a strongly continuous semigroup and in terms of a resolvent estimate.
\end{abstract}

\maketitle

\section{Introduction}

In this note we study
 operators $C$ related to linear observation systems
\begin{equation}\label{eq:observationsystem}
\left\{
\begin{array}{lll}
\dot{x}(t) &= Ax(t), \ t \geq 0, \\
x(0) &= x_0,\\
y(t) &= Cx(t),\\
\end{array}\right.
\end{equation}
where $A$ generates a $C_0$-semigroup $(T(t))_{t \geq 0}$ on a Banach space $X$, the state space and $x_0 \in X$ is the initial state. The operator $C$ is the output operator (or ``observation operator'') mapping the domain of $A$ (equipped with the graph norm) continuously into some Banach space $Y$, the output space. Since the (mild) solution $x$ in \eqref{eq:observationsystem} 
is given by the semigroup applied to the initial state, 
$x(\cdot)=T(\cdot)x_0$, 
the output is formally given by
\[y(\cdot) = CT(\cdot)x_0.\]
Admissibility with respect to $\Lp^p$ (or $\Lp^p$-admissibility) means that the mapping $x_0 \mapsto CT(\cdot)x_0$, initially defined only on $D(A)$, has a continuous extension from $X$ to $\Lp^p(0,\tau;Y)$. Here we distinguish between finite-time admissibility (or just admissibility), that is if $\tau <\infty$ and infinite-time admissibility, that is if $\tau=\infty$. In the literature admissibility with respect to various spaces of output functions, different than $\Lp^p$, has been considered, see e.g. \cite{HaakKunst07, JNPS16, Wynn11}. In this introduction we focus on infinite-time admissibility, noting that all mentioned results can be formulated analogously for finite-time admissibility. Indeed, the infinite-time case is more interesting and the finite-time case can be derived by the infinite-time case since $\Lp^p$-admissibility is invariant under scaling of the semigroup and hence under shifting of the generator.
The notion of admissibility is strongly related to the more general concept of well-posed linear systems being of central interest in system theory, see e.g. \cite{TucWei14} for $\Lp^2$.

In this article we study two characterizations of this notion with respect to Orlicz spaces $\Lp_\Phi$ and $\E_\Phi$; an estimate on the resolvent of  the Laplace transformed trajectory of the output $y=CT(\cdot)$ and an equivalence in terms of generation of a strongly continuous semigroup. These results extend earlier results for the special case of $L^p$-admissibility due to Callier--Grabowski \cite{CallGrab96}, Engel \cite{Engel.Char.Adm},  as well as Le Merdy \cite{LeMerdy2003} and Haak \cite{Haakthesis}. 
\begin{mainresult}[c.f.~\Cref{thm:Engel_generalization}]\label{main:Engel}
Let $C:D(A) \rightarrow Y$ be a bounded operator with respect to the graph norm of $A$. Then $C$ is infinite-time $\E_\Phi$-admissible if and only if the block operator matrix
\begin{equation}
\mathcal{A} =
 \begin{pmatrix}
A & 0 \\
0 & -\frac{\d}{\d r}
\end{pmatrix}
\begin{pmatrix}
I & 0 \\
L & I \\
\end{pmatrix}
\end{equation}
with domain
\begin{equation}
D(\mathcal{A}) = \left\{ \begin{pmatrix}
x\\
f\\
\end{pmatrix}
\in D(A) \times \W^1\E_\Phi(0,\infty;Y) \, \middle\vert  \, Cx + f(0) = 0  \right\}
\end{equation}
generates a $C_0$-semigroup on $X \times \E_\Phi(0,\infty;Y)$, where $Lx:= \mathds{1}_{\vert_{[0,\infty)}} Cx$ for $x \in D(A)$.
\end{mainresult}
Here $\W^1\E_\Phi(0,\infty;Y)$ denotes the (standard) Orlicz-Sobolev space consisting of those functions in $\E_\Phi(0,\infty;Y)$ whose weak derivative exists and lies again in $\E_\Phi(0,\infty;Y)$, see e.g.~\cite[Ch.~8]{Adams}.
\bigskip

A second characterization of $\Lp^p$-admissibility due to Le Merdy ($p=2$) and Haak ($p \geq 1)$ relates to a conjecture originally formulated by Weiss in \cite{We90} (for $p=2$ and, equivalently, for the dual problem of control operators). The $p$-Weiss conjecture states that infinite-time $\Lp^p$-admissibility of $C$ is equivalent to the so-called {\em infinite-time $p$-Weiss condition} for $C$, i.e.
\begin{equation}\label{eq:p_Weiss_condition}
\sup_{\real z> 0}(\real z)^{1-\nicefrac{1}{p}} \lVert C(z-A)^{-1} \rVert < \infty,
\end{equation}
a property which is easily seen to follow from $\Lp^p$-admissibility by H\"older inequality.
The question thus is whether the $p$-Weiss condition is sufficient for $\Lp^p$-admissibility of $C$. Whereas the answer is negative in the general Banach space setting \cite{We90} (and $p=2$), the problem has received much attention since then, with both positive results, as well as counterexamples. We mention here some of them and refer to \cite{Haakthesis} and \cite{jpsurvey} for detailed overviews. In \cite{JPP02,JaZw04} it is shown, that the $2$-Weiss conjecture does not hold in arbitrary Hilbert spaces without further assumptions on the semigroup and the operators $C$. For Hilbert spaces the $p=2$-case is known to hold true for exponentially stable, left-invertible semigroups, \cite{We90}, 
 as well as in the case of contraction semigroups and finite-dimensional output spaces, \cite{JaPa01}. For infinite-dimensional output spaces, the statement may fail even for semigroups of isometries \cite{JPP02}. 
Le Merdy showed in \cite{LeMerdy2003} that the $2$-Weiss conjecture holds true in the Hilbert space situation under the assumption of an analytic contractive semigroup. Moreover, he showed for Banach spaces and a bounded analytic semigroup that the $2$-Weiss conjecture holds if and only if the operator $(-A)^{\nicefrac{1}{2}}$ defined via the holomorphic functional calculus is infinite-time $\Lp^2$-admissible.
Haak \cite{Haakthesis} extended Le Merdy's results to more general $p\ge1$  as follows:
If $A$ generates a bounded analytic semigroup and $A$ has dense range, then the $p$-Weiss conjecture holds if and only if $(-A)^{\nicefrac{1}{p}}$ is infinite-time $\Lp^p$-admissible.
He used generalized square function estimates for the operator $A$ which are equivalent to $(-A)^{\nicefrac{1}{p}}$ being infinite-time $\Lp^p$-admissible. 
 
 We continue these developments in the context of Orlicz spaces for a certain class of Young functions, which is called {\it class $\mathcal{P}$}. Our second main result reads:

\begin{mainresult}[c.f.~\Cref{thm:Phi_Weiss_conjecture}]\label{main:Weiss}
Suppose $A$ generates a bounded analytic semigroup, the Young function $\Phi$ is of class $\mathcal{P}$ and either
\begin{enumerate}[label=$\bullet$]
\item $A$ is a multiplication operator with $\sigma(-A) \subseteq [0,\infty)$ or 
\item $\Phi^{-1}$ is holomorphic on some sector containing the spectrum of $-A$ and there exist constants $m_0,m_1>0$ such that for $z$ in this sector it holds that 
\[m_0 \Phi^{-1}(\lvert z \rvert) \leq  \lvert \Phi^{-1}(z) \rvert \leq m_1 \Phi^{-1}(\lvert z \rvert).\]
\end{enumerate}
Then, for a bounded operator $C:D(A) \rightarrow Y$ the following are equivalent
\begin{enumerate}[label=\rm (\roman*)]
\item  $\Phi^{-1}(-A)$ is infinite-time $\Lp_\Phi$-admissible,
\item It holds that
\begin{equation*}
C \text{ is infinite-time } \Lp_\Phi \text{-admissible} \Leftrightarrow 
\sup_{\real z >0} \widetilde{\Phi}^{-1}(\real z) ~ \lVert C(z-A)^{-1} \rVert < \infty.
\end{equation*}
\end{enumerate}
Here, $\widetilde{\Phi}$ denotes the to $\Phi$ complementary Young function, see \Cref{sec:Orlicz_spaces}.
\end{mainresult}

This paper is structured as follows. In \Cref{sec:Orlicz_spaces} we recall the for us necessary definitions and facts on Young functions $\Phi$ and Orlicz spaces $\Lp_\Phi$ and $\E_\Phi$. Furthermore, we introduce a subclass $\mathcal{P}$ of Young functions with polynomial behavior at $0$ and $\infty$ which play a role when dealing with the functional calculus in the context of the Weiss conjecture. We shift more technical results on Orlicz spaces to the Appendix, where the reader can also find results on the dual space of vector valued Orlicz spaces, the shift semigroup on Orlicz spaces and a generalization of Minkowski's integral inequality to Orlicz spaces used in the ensuing sections. 
In \Cref{sec:Engel} we formalize Orlicz admissibility for observation operators and prove \Cref{main:Engel}, which is an adoption of the $\Lp^p$ case by our previous investigations on the shift semigroup in Orlicz spaces. \Cref{sec:Orlicz-Weiss} is devoted to the extension of the Weiss conjecture for Orlicz spaces and the proof of \Cref{main:Weiss}. Furthermore, we give a sufficient condition for infinite-time $\Lp_\Phi$-admissibility of $\Phi^{-1}(-A)$ in case of a semigroup on $\ell^r$, $1 \leq r < \infty$ generated by a multiplication operator. In \Cref{sec:Duality} we discuss the duality concept of Orlicz admissible control and observation operators which allows to transfer all results on Orlicz admissibility to control operators.

In the following $X$ and $Y$ are Banach spaces and $A$ is the generator of a $C_0$-semigroup $(T(t))_{t\geq0}$ with growth bound $\omega(A)$. The domain of $A$ equipped with the graph norm is denoted by $D(A)$. The spectrum of $A$ is $\sigma(A)$ and its resolvent is $R(z,A)$ for $z$ in the resolvent set $\rho(A)$. The set of all bounded operators from $D(A)$ to $Y$ is given by $\mathcal{L}(D(A),Y)$. As usually, $I=I_X$ denotes the identity operator on the space $X$.

\section{Orlicz spaces}\label{sec:Orlicz_spaces}
We recall Young functions and Orlicz spaces. For details see \cite{Adams, KrasnRut, Kufner}.

\begin{definition}
A strictly increasing continuous and convex function $\Phi:(0,\infty) \rightarrow (0,\infty)$ is called {\em Young function} if
\[ \lim_{t \to 0} \frac{\Phi(t)}{t}=0 \quad \text{ and } \quad \lim_{t \to \infty} \frac{\Phi(t)}{t}= \infty.\]
\end{definition}

Such a $\Phi$ is invertible and extends continuously to $[0,\infty)$ by $\Phi(0)=0$.

\begin{definition}
Let $\Phi$ be a Young function. Then, the function $\widetilde{\Phi}$ defined by
\[ \widetilde{\Phi}(s) = \sup_{t \geq 0} (st - \Phi(t)) \]
is called the {\em to $\Phi$ complementary Young function} and $\Phi$ and $\widetilde{\Phi}$ are called {\em complementary to each other}.
\end{definition}
Indeed, $\widetilde{\Phi}$ defines a Young function and $\Phi(t)=\sup_{s \geq 0} (st - \widetilde{\Phi}(s))$.
For complementary Young functions $\Phi$ and $\widetilde{\Phi}$ it holds that (see~\cite{Adams})
\begin{equation}\label{eq:behavior_Phi^-1*tildePhi^-1}
t \leq \Phi^{-1}(t) \widetilde{\Phi}^{-1}(t) \leq 2t \quad \text{ for all } t \geq 0.
\end{equation}

Let $(\Omega,\mathcal{F},\mu)$ be a measure space and $Y$ be a Banach space. Whenever $\Omega$ is an interval in $\R$, $(\mathcal{F},\mu)$ refers to the Borel $\sigma$-algebra and the Lebesgue measure.

\begin{definition}
Let $\Phi$ be a Young function. The Orlicz space $\Lp_\Phi(\Omega;Y)$ is given by
\begin{align*}
\Lp_\Phi(\Omega;Y) \coloneqq \Big\{ f:\Omega \rightarrow Y \: \Big\vert \:& f \text{ is Bochner measurable and}\\
&\text{ for some } k>0: \: \int_\Omega \Phi\left( \frac{\lVert f(\cdot) \rVert_Y}{k} \right) \d \mu < \infty \Big\},
\end{align*}
where we identify functions which coincide almost everywhere. Further, we define the so-called {\em Luxemburg norm} for $f\in \Lp_\Phi(\Omega,Y)$ by
\begin{equation*}
\lVert f \rVert_{\Lp_\Phi(\Omega;Y)} \coloneqq \inf\left\{ k>0 \: \middle\vert \: \int_\Omega \Phi\left( \frac{\lVert f(\cdot) \rVert_Y}{k} \right) \d \mu \leq 1 \right\}.
\end{equation*}
We define a subspace of $\Lp_\Phi(\Omega;Y)$ by
\begin{equation*}
\E_\Phi(\Omega,Y) \coloneqq \overline{\{ f \in \Lp^\infty(\Omega;Y) \mid f \text{ has bounded essential support} \}}^{\lVert \cdot \rVert_{\Lp_\Phi}}. 
\end{equation*}
We may write $\Lp_\Phi$ and $\E_\Phi$ if $\Omega$ and $Y$ are clear from the context.\\
On $\Lp_\Phi(\Omega;Y)$ we can define the {\em Orlicz norm} by
\begin{align*}\label{eq:def_Orlicz_norm}
\interleave f \interleave_{\Lp_\Phi(\Omega;Y)} 
&\coloneqq
\sup \set{\left\lvert \int_\Omega \langle f, g \rangle_{Y,Y'} \d \mu \right\rvert}{ \lVert g \rVert_{\Lp_{\widetilde{\Phi}}(\Omega;Y')} \leq 1}\\
&=\sup \set{ \int_\Omega \left\lvert \langle f, g \rangle_{Y,Y'} \right\rvert \d \mu }{\lVert g \rVert_{\Lp_{\widetilde{\Phi}}(\Omega;Y')} \leq 1},
\end{align*}
where the second equality can be verified by considering $\tilde{g}=g \sgn(\langle f, g \rangle_{Y,Y'})$.
\end{definition}
ince $\lVert g \rVert_{\Lp_{\widetilde{\Phi}}(\Omega;Y')} \leq 1$ if and only if $\int_\Omega \widetilde{\Phi}(\lVert g \rVert_Y') \d \mu \leq 1$, the Orlicz norm can be interpreted as dual norm to the Luxemburg norm of the complementary Young functions. Both norms are equivalent, as for $f \in \Lp_\Phi(\Omega;Y)$ it holds that
\begin{equation}\label{eq:eqivalent_norms_Luxemburg_Orlicz}
\lVert f \rVert_{\Lp_\Phi(\Omega;Y)} \leq \interleave f \interleave_{\Lp_\Phi(\Omega;Y)} \leq  2 \lVert f \rVert_{\Lp_\Phi(\Omega;Y)}.
\end{equation}
Indeed, the inequality is well-known for the scalar valued case, see e.g.~\cite{Kufner}. It is less obvious for vector valued Orlicz spaces as the fact that $\interleave f \interleave_{\Lp_\Phi(\Omega;Y)} = \interleave \, \lVert f \rVert_Y \, \interleave_{\Lp_\Phi(\Omega;\R)}$ is non-trivial. Both proofs can be found in the Appendix. Finally, $\Lp_\Phi(\Omega;Y)$ and $\E_\Phi(\Omega;Y)$ equipped with these norms are complete. \medskip

An important property of Young functions is the so-called $\Delta_2$-condition. 
\begin{definition}
A Young function $\Phi$ satisfies the {\em $\Delta_2$-condition near infinity} ($\Phi \in \Delta_2^\infty$) if there exists constants $K,t_0 \geq 0$ such that
\begin{equation}\label{eq:Delta_2}
\Phi(2t) \leq K \Phi(t) \quad \text{ for all } t \geq t_0.
\end{equation}
We say that $\Phi$ satisfies the {\em $\Delta_2$-condition globally} ($\Phi \in \Delta_2^{\text{\tiny global}}$) if  \eqref{eq:Delta_2} holds with $t_0=0$.
\end{definition}

It is known that $E_\Phi(\Omega;Y) = \Lp_\Phi(\Omega;Y)$ if and only if $\Phi$ satisfies the $\Delta_2$-condition globally if $\mu(\Omega)=\infty$ or near infinity if $\mu(\Omega)<\infty$. Similar as for $\Lp^p$ spaces the {\em generalized H\"older inequality}
\begin{equation}\label{eq:hoelder}
\int_\Omega \lVert f(\cdot) \rVert_{Y_1} ~ \lVert g(\cdot) \rVert_{Y_2} \d \mu \leq 2 \lVert f \rVert_{\Lp_\Phi(\Omega;Y_1)} \lVert g \rVert_{\Lp_{\widetilde{\Phi}}(\Omega;Y_2)}
\end{equation}
holds for all $f \in \Lp_\Phi(\Omega;Y_1)$, $g \in \Lp_{\widetilde{\Phi}}(\Omega;Y_2)$ and Banach spaces $Y_1,Y_2$,

The Young function $\Phi(t)=t^p$ with $1<p<\infty$ satisfies the $\Delta_2$-condition globally. It leads to the Orlicz space $\Lp_\Phi=\E_\Phi=\Lp^p$ and the Luxemburg norm is just the standard norm on $\Lp^p$. Further, the to $\Phi$ complementary Young function is (up to constants) given by $\widetilde{\Phi}(s)=s^{p'}$ with $\frac{1}{p} + \frac{1}{p'}=1$ and hence $\Lp_{\widetilde{\Phi}}=\E_{\widetilde{\Phi}}=\Lp^{p'}$. Thus, Orlicz spaces and the concept of complementary Young functions generalize $\Lp^p$ spaces and H\"older conjugates in a natural way.\bigskip

We close the section by introducing a subclass of Young functions which play a role in \Cref{sec:Orlicz-Weiss} in the context of the holomorphic functional calculus.

\begin{definition}\label{def:class_P}
We say that a function $\Phi:(0,\infty) \rightarrow (0,\infty)$ is of class $\mathcal{P}$ ($\Phi \in \mathcal{P}$), if $\Phi$ is invertible and 
\begin{equation}\label{eq:Phi^-1calssP}
\Phi^{-1}(t) = t^{\nicefrac{1}{p}} \rho(t^{\nicefrac{1}{q}-\nicefrac{1}{p}}) \quad \text{ with } 1 < p < q < \infty
\end{equation}
for $t>0$, where $\rho:(0,\infty) \rightarrow (0,\infty)$ is a continuous concave function such that 
\begin{equation}\label{eq:property_rho}
\rho(st) \leq \max(1,s) \rho(t)  \quad \text{ for all } s,t>0.
\end{equation}
\end{definition}

\begin{remark}
\begin{enumerate}[label= \arabic*.]
\item By \Cref{apx:P_functions_are_Young_functions}, functions of class $\mathcal{P}$ are Young functions.
\item In \cite{karlovich2001interpolation} the authors state the following result on the representation of $\Phi$ and $\widetilde{\Phi}$, see \cite[Lem.~3.2]{karlovich2001interpolation}: If $\Phi \in \mathcal{P}$  is characterized by \eqref{eq:Phi^-1calssP}, then
\begin{equation}\label{eq:PhiclassP}
\Phi(t) = t^q h(t^{p-q})
\end{equation}
and
\begin{equation}\label{eq:tildePhiclassP}
\widetilde{\Phi}(t) = t^{p'} k(t^{q'-p'})
\end{equation}
where $p'$ and $q'$ are the H\"older conjugates to $p$ and $q$ and $h,k:(0,\infty)\rightarrow(0,\infty)$ are continuous quasi-concave functions such that $h(t)>0$ for $t>0$ and $h(st) \leq \max(1,s) \rho(t)$ for all $s,t>0$ (for $k$ the same). The functions $h$ and $k$ are defined by \eqref{eq:PhiclassP} and \eqref{eq:tildePhiclassP}.
\item From \eqref{eq:Phi^-1calssP}, \eqref{eq:PhiclassP} and \eqref{eq:tildePhiclassP} we derive
\begin{align}\label{eq:ineq_for_class_P_direct}
\begin{split}
\Phi^{-1}(st) &\leq  \max(s^{\nicefrac{1}{p}},s^{\nicefrac{1}{q}} )~ \Phi^{-1}(t),\\
\Phi(st) &\leq  \max(s^q,s^{p})~ \Phi(t),\\
\widetilde{\Phi}(st) &\leq  \max(s^{p'},s^{q'})~ \widetilde{\Phi}(t)
\end{split}
\end{align}
for $s,t>0$ and with the transformations $u=\max(s^q,s^{p})$, $v=\Phi(t)$  and $u=s^{p'} \max(s^{p'},s^{q'})$, $v=\widetilde{\Phi}(t)$
\begin{align}\label{eq:ineq_for_class_P_transformed}
\begin{split}
\min(u^{\nicefrac{1}{p}},u^{\nicefrac{1}{q}}) ~\Phi^{-1}(v) &\leq \Phi^{-1}(uv),\\
\min(u^{\nicefrac{1}{q'}},u^{\nicefrac{1}{p'}})~ \widetilde{\Phi}^{-1}(v) &\leq \widetilde{\Phi}^{-1}(uv).
\end{split}
\end{align}
In particular we infer by \eqref{eq:ineq_for_class_P_direct} that $\Phi,\widetilde{\Phi} \in \Delta_2^{\text{\tiny global}}$.
\end{enumerate}
\end{remark}

\begin{example}\label{example:class_P}
\begin{enumerate}[label=\rm (\roman*)]
\item If $\rho,\mu: (0,\infty) \rightarrow (0,\infty)$ are continuous concave functions satisfying \eqref{eq:property_rho}, then so are $a \rho + b \mu$ and  $\rho \circ \mu$ for $a,b>0$. To see that $\rho \circ \mu$ satisfies \eqref{eq:property_rho}, note that $\rho$ is increasing by \eqref{eq:property_rho}.
\item The trivial examples $\rho_r(t)= t^r$ for some $r \in[0,1]$ lead to the Young functions $\Phi(t)=t^\alpha$ with $\alpha  \in [p,q]$ (depending on $r$) when $\Phi^{-1}$ is given by \eqref{eq:Phi^-1calssP} with $1 < p< q<\infty$. If $r=0$, then $\alpha=p$ and if $r=1$, then $\alpha=q$. Indeed, $\alpha$ is given by $\frac{1}{\alpha} = \frac{r}{q} + \frac{1-r}{p} \in [\frac{1}{q},\frac{1}{p}]$. \\
For $r=0$ and $r=1$ the corresponding functions $\rho_0(t)=1$ and $\rho_1(t)=t$ can be seen as the extreme cases for $\rho$ when talking about the slope of increasing concave functions.
\item The following example can be found in \cite{karlovich2001interpolation}. Let $\Phi^{-1}$ be given by \eqref{eq:Phi^-1calssP} with $\rho(t)=\min(1,t)$, $t \geq 0$ and any choice of $1<p<q<\infty$. Then $\Phi$ is given by \eqref{eq:PhiclassP} with $h(t)=\max(1,t)$, $t \geq 0$. It is obvious that $\Phi$ is of class $\mathcal{P}$.
\item Let $\Phi^{-1}$ be given by \eqref{eq:Phi^-1calssP} with $\rho(t)=\log(1+t)$, $t \geq 0$ and any choice of $1<p<q<\infty$. Then, $\Phi$ is of class $\mathcal{P}$, $\Phi^{-1}$ has a holomorphic extension to any sector $S_\delta \coloneqq \{ z \in \C \setminus \{0\} \mid \lvert \arg z \rvert < \delta\} $ (taking the principal branch of the complex logarithm) and for $\delta \leq \frac{\pi}{3}$ then there exist constants $m_0,m_1>0$ such that
\begin{equation*}
m_0 \Phi^{-1}(\lvert z \rvert) \leq  \lvert \Phi^{-1}(z) \rvert \leq m_1 \Phi^{-1}(\lvert z \rvert)
\end{equation*} for $z \in S_\delta$. The technical details of this fact are shifted to the Appendix.
\end{enumerate}
\end{example}

\section{Characterizing Orlicz admissibility by semigroups}\label{sec:Engel}

We recall the definition of admissibility already indicated in the introduction.

\begin{definition}\label{def:adm}
An operator $C \in \Lb(D(A),Y)$ is called {\em $\Lp_\Phi$-admissible} for $A$ (or equivalently for $(T(t))_{t \geq 0}$) if for some (and hence for all) $\tau>0$ there exists a minimal constant $c_\tau>0$ such that
\begin{equation}\label{eq:adm}
\lVert CT(\cdot)x \rVert_{\Lp_\Phi(0,\tau;Y)} \leq c_\tau \lVert x \rVert_X \quad \text{ for all } x \in D(A).
\end{equation}
If, additionally,  $c_\infty \coloneqq \sup_{\tau \geq 0}c_\tau < \infty$, then $C$ is called {\em infinite-time $\Lp_\Phi$-admissible} for $A$. If there is no ambiguity, we will simply say that 
``$C$ is (infinite-time) $\Lp_\Phi$-admissible'' without referring to the semigroup.
Analogously, we define (infinite-time) $\E_\Phi$-admissibility.
\end{definition}
It is not hard to show that $\Lp_\Phi$-admissibility is invariant under scaling of the semigroup, i.e. $C$ is  $\Lp_\Phi$-admissible for $(T(t))_{t \geq 0}$ if and only if $C$ is $\Lp_\Phi$-admissible for the rescaled semigroup $(\e^{-\varepsilon t}T(t))_{t \geq 0}$  for any $\varepsilon \in \R$. Also note that if the semigroup is exponentially stable, then finite-time and infinite-time $\Lp_\Phi$-admissibility are equivalent. Moreover, $C$ is infinite-time $\Lp_\Phi$-admissible if and only if \eqref{eq:adm} holds for $\tau= \infty$ and a constant $c_\infty<\infty$. The same holds for $\E_\Phi$-admissibility. \medskip

The main result of this section is a generalization of a result by Callier--Grabowski \cite{CallGrab96}, see also Engel \cite{Engel.Char.Adm} for a shorter proof, from $\Lp^p$ to $\E_\Phi$.

\begin{theorem}\label{thm:Engel_generalization}
The operator $C \in \Lb(D(A),Y)$ is (infinite-time) $\E_\Phi$-admissible if and only if for some (and hence for all) $\tau>0$ ($\tau=\infty$ in the infinite-time case) the block operator matrix
\begin{equation}\label{eq:block_op_for_C}
\mathcal{A} =
 \begin{pmatrix}
A & 0 \\
0 & -\frac{\d}{\d r}
\end{pmatrix}
\begin{pmatrix}
I & 0 \\
L & I \\
\end{pmatrix}
\end{equation}
with domain
\begin{equation}\label{eq:domain_block_op_for_C}
D(\mathcal{A}) = \left\{ \begin{pmatrix}
x\\
f\\
\end{pmatrix}
\in D(A) \times \W^1\E_\Phi(0,\tau;Y) \, \middle\vert \, Cx+ f(0) = 0 \right\}
\end{equation}
generates a $C_0$-semigroup on $X \times \E_\Phi(0,\tau;Y)$, where $Lx:= \mathds{1}_{\vert_{[0,\tau]}} Cx$ for $x \in D(A)$.
\end{theorem}

The proof of \Cref{thm:Engel_generalization} follows Engel's approach, relying on the following consequence from \cite{Engel.Op.Matr}, and the fact that $-\frac{\d}{\d r}$ with domain $\{ f \in \W^1\E_\Phi(0,\tau;Y) \mid f(0)=0 \}$ generates the right shift semigroup on $\E_\Phi(0,\tau;Y)$, see~\Cref{apx:rightshift_EPhi}.

\begin{lemma}\label{lem:Engel_generator}
If $X$ and $F$ are Banach spaces, $A:D(A) \subseteq X \rightarrow X$ and $D:D(D) \subseteq F \rightarrow F$ are closed and densely defined operators such that $(\omega,\infty) \subseteq \rho(A) \cap \rho(D)$ for some $\omega \in \R$ and if $L \in \Lb(D(A),F)$, then the following are equivalent
\begin{enumerate}[label=(\alph*)]
\item The block operator matrix
\begin{equation*}
\mathcal{A} = \begin{pmatrix}
A & 0 \\
0 & D \\
\end{pmatrix}
\begin{pmatrix}
I & 0 \\
L & I \\
\end{pmatrix}
\end{equation*}
with domain
\begin{equation*}
D(\mathcal{A}) = \left\{ \begin{pmatrix}
x \\
f \\
\end{pmatrix} \in D(A) \times F \, \middle\vert \, Lx + f \in D(D) \right\}
\end{equation*}
generates a $C_0$-semigroup on $X \times F$,
\item $A$ generates a $C_0$-semigroup $(T(t))_{t \geq 0 }$ on $X$, $D$ generates a $C_0$-semigroup $(S(t))_{t \geq 0}$ on $F$ and for some (and hence for all) $t_0>0$ it holds that $\sup_{t \in [0,t_0] } \lVert R(t) \rVert_{\Lb(X,F)}< \infty$, where $R(t)$ is the bounded extension of the operator
\begin{equation*}
R(t)x= D \int_0^t S(t-s) L T(s)x \d s , \quad x \in D(A^2).
\end{equation*}
\end{enumerate}
If one of the equivalent conditions is satisfied, the semigroup generated by $\mathcal{A}$ is given by
\begin{equation*}
T_{\mathcal{A}}(t) = \begin{pmatrix}
T(t) & 0 \\
R(t) & S(t) \\
\end{pmatrix}.
\end{equation*}
\end{lemma}
\begin{proof}[Proof of \Cref{thm:Engel_generalization}]
\Cref{lem:Engel_generator} for $D=-\frac{\d}{\d r}$ with domain $\{ f \in \W^1\E_\Phi(0,\tau;Y) \mid f(0)=0 \}$ and $F=\E_\Phi(0,\tau;Y)$ yields that $\mathcal{A}$ given by \eqref{eq:block_op_for_C} and \eqref{eq:domain_block_op_for_C} generates a $C_0$-semigroup on $X \times \E_\Phi(0,\tau;X)$ if and only if $\sup_{t \in [0,\tau]} \lVert R(t) \rVert < \infty$. Since $D$ generates the right shift semigroup $(S(t))_{t \geq 0}$ on $\E_\Phi(0,\tau;Y)$ we conclude that $R(t)$ is given by
\begin{align*}
[R(t)x](r) &=
- \frac{\d}{\d r} \int_0^t S(t-s) \mathds{1}_{\vert_{[0,t_0]}}(r) CT(s)x \d s \\
&= - \frac{\d}{\d r} \int_{\max\{0,t-r\}}^t CT(s)x \d s \\
&= - \mathds{1}_{\vert_{[0,t]}}(r) CT(\max\{0,t-r\})x
\end{align*}
for $x \in D(A^2)$ and $r \in [0,\tau]$. Thus, $\mathcal{A}$ generates a $C_0$-semigroup on $X \times  \E_\Phi(0,\tau;Y)$ (with $\tau=\infty)$ if and only if $C$ is (infinite-time) $\E_\Phi$-admissible.
\end{proof}

\section{On the Weiss conjecture for Orlicz spaces}\label{sec:Orlicz-Weiss}


We start by introducing some notation and assumptions in order to generalize Haak's result on the $p$-Weiss conjecture as stated in the introduction, see also \cite{Haakthesis}, to Orlicz spaces $\Lp_\Phi$.

We denote by $\C_\alpha$, $\alpha \in \R$, and $S_\delta$, $\delta \in (0,\frac{\pi}{2})$ the open right half-plane with abcissa $\alpha$ and the open sector with opening angle $2\delta$, i.e.
\[\C_\alpha = \{ z \in \C \mid \real z > \alpha \} \quad \text{and}\quad S_\delta = \{ z \in \C \setminus \{0\} \mid  \lvert \arg z \rvert < \delta\}.\] 
For $\delta =0$ we write $S_0=(0,\infty)$.\medskip


We recall that $A$ is the generator of a bounded analytic semigroup if $-A$ is densely defined and  {\em sectorial of type $\omega$} for some $\omega \in [0,\frac{\pi}{2})$. The latter means that $\sigma(-A) \subset \overline{S_\omega}$ and for every $\delta \in (\omega,\pi)$ there is a constant $M_\delta>0$ such that
\[ \lVert z R(z,-A) \rVert \leq M_\delta \quad \text{ for all } z \in \C \setminus \overline{S_\delta}.\] 
For more about sectorial operators, the holomorphic functional calculus and the connection to bounded analytic semigroups we refer to \cite{Ha06} and \cite{engelnagel}. Whenever we are dealing with the generator $A$ of a bounded analytic semigroup $\omega_A$ is referred to as the {\em angle of sectoriality} of $-A$, that is the infimum of all $\omega$ such that $-A$ is sectorial of type $\omega$. Besides the holomorphic functional calculus, the functional calculus for multiplication operators (see~\cite[Ch.~2]{isem21}) will be of interest.

The assumption of having a bounded analytic semigroup is necessary for Haak's result in the sense that if $(-A)^{\nicefrac{1}{p}}$ is $\Lp^p$-admissible and the semigroup is bounded then it is already bounded analytic, \cite[Prop.~2.7]{hamid2010direct}. 

When passing from $\Lp^p$ to $\Lp_\Phi$ we have to substitute both the $p$-Weiss condition \eqref{eq:p_Weiss_condition} and the operator $(-A)^{\nicefrac{1}{p}}$ in terms of the Young function $\Phi$ or its complementary Young function $\widetilde{\Phi}$. While the former is rather straight-forward, the latter is more delicate. First we generalize the $p$-Weiss condition \eqref{eq:p_Weiss_condition}.



\begin{definition}\label{def:Phi_Weiss_condition}
An operator $C \in \Lb(D(A),Y)$ is said to satisfy the {\em $\Phi$-Weiss condition}, for a Young function $\Phi$, if
\begin{equation}\label{eq:Phi_Weiss_condition}
\sup_{z \in \C_\alpha} \left(\lVert \e^{-\real z \cdot} \rVert_{\Lp_{\widetilde{\Phi}}(0,\infty)} \right)^{-1} \lVert C R(z,A) \rVert < \infty
\end{equation}
for some $\alpha>0$, where $\widetilde{\Phi}$ is the complementary Young function of $\Phi$. We say that $C$ satisfies the {\em infinite-time $\Phi$-Weiss condition} if \eqref{eq:Phi_Weiss_condition} holds for $\alpha =0$.
\end{definition}

\begin{remark}
If $\widetilde{\Phi} \in \Delta_2^{\text{\tiny global}}$ then we can replace $(\lVert \e^{- \real z \cdot} \rVert_{\Lp_{\widetilde{\Phi}}(0,\infty})^{-1}$ by $\widetilde{\Phi}^{-1}(\real z)$, see~\Cref{apx:Orlicz_norm_exp_function}. Recall that if $\Phi \in \mathcal{P}$ then $\widetilde{\Phi} \in \Delta_2^{\text{\tiny global}}$. It is obvious that the definitions of the $\Phi$-Weiss condition and the $p$-Weiss condition \eqref{eq:p_Weiss_condition} are consistent in the sense that they are the same if we consider $\Phi(t)=t^p$ for $1<p<\infty$.
\end{remark}

Similar to $\Lp^p$ spaces it is easy to prove that (infinite-time) $\Lp_\Phi$-admissibility of $C \in \Lb(D(A),Y)$ implies the (infinite-time) $\Phi$-Weiss condition. 
 
\begin{lemma}\label{lem:adm_implies_Weiss_condition}
If $C \in \Lb(D(A),Y)$ is (infinite-time) $\Lp_\Phi$-admissible, then the (infinite-time) $\Phi$-Weiss condition holds.
\end{lemma}

\begin{proof}
If $C$ is infinite-time $\Lp_\Phi$-admissible. Using the generalized H\"older inequality \eqref{eq:hoelder} we get for all $x \in D(A)$ that
\begin{align*}
\lVert C R(z,A)x \rVert
&= \left \lVert \int_0^\infty \e^{-zt} CT(t)x \d t \right \rVert \\
&\leq 2 \lVert \e^{- \real z \cdot} \rVert_{\Lp_{\widetilde{\Phi}}(0,\infty)} \lVert CT(\cdot)x \rVert_{\Lp_\Phi(0,\infty;Y)} \\
&\leq 2 c_\infty \lVert \e^{- \real z \cdot} \rVert_{\Lp_{\widetilde{\Phi}}(0,\infty)} \lVert x \rVert
\end{align*} 
holds, where $c_\infty$ denotes the admissibility constant from \Cref{def:adm}. If $C$ is $\Lp_\Phi$-admissible for $(T(t))_{t \geq 0}$, then $C$ is infinite-time $\Lp_\Phi$-admissible for the exponentially stable semigroup generated by $A-\alpha$, where $\alpha>\omega(A)$ and $\alpha>0$. Hence, the proof can be deduced from the infinite-time case.
\end{proof}

{If $\Lp_\Phi=\Lp^p$ Haak's result tells us that the converse of \Cref{lem:adm_implies_Weiss_condition} holds if and only if $\Phi^{-1}(-A)=(-A)^{\nicefrac{1}{p}}$ is (infinite-time) $\Lp^p$-admissible, hence formally $\Phi^{-1}(-A)$ seems to be a suitable operator to characterize general $\Lp_\Phi$-admissibility. However, we have to make sure that this operator is indeed a well-defined in $\Lb(D(A),X)$. Therefore, we make the following assumption on $\Phi$.

\begin{assumption}\label{assump:Phi^-1_holomorph}
Assume that $\Phi^{-1}$ extends to a holomorphic function on some sector $S_\delta$ for $\delta \in (\omega_A,\frac{\pi}{2})$ and that there exist constants $m_0,m_1>0$ such that 
\[m_0 \Phi^{-1}(\lvert z \rvert) \leq  \lvert \Phi^{-1}(z) \rvert \leq m_1 \Phi^{-1}(\lvert z \rvert) \quad \text{ for all } z\in S_\delta.\]
\end{assumption}

We continue with a sequence of technical statements on the operator $\Phi^{-1}(-A)$ and the $\Phi$-Weiss condition. Our approach is based on the ideas from \cite{hamid2010direct}, which seem to be slightly more elementary than the more natural proof of Haak using square function estimates. It seems to be a non-trivial challenge to generalize such square function estimates to the Orlicz space setting.\medskip


Recall that a multiplication operator is an operator $M_a:\Lp^p(\Omega)\rightarrow \Lp^p(\Omega)$ for some semi-finite measure space $(\Omega,\mathcal{F},\mu)$, $1 \leq p \leq \infty$ and $a:\Omega \rightarrow \C$ measurable, given by
\begin{equation*}
M_af \coloneqq af \qquad \text{ for } \qquad f \in D(M_a) \coloneqq \{ f \in \Lp^p(\Omega) \mid a f \in \Lp^p(\Omega) \}.
\end{equation*}

\begin{lemma}
Suppose that $A$ generates a bounded analytic semigroup. If either
\begin{enumerate}[label=\rm (\roman*)]
\item $A$ is a multiplication operator with $\sigma(-A)\subseteq[0,\infty)$ or 
\item \Cref{assump:Phi^-1_holomorph} holds  (for $\Phi^{-1})$ and additionally $\Phi \in \mathcal{P}$
\end{enumerate}
then $\Phi^{-1}(-A)$ is well-defined via the functional calculus for multiplication operators and holomorphic functional calculus, respectively, and $\Phi^{-1}(-A) \in \Lb(D(A),X)$.
\end{lemma}

\begin{proof}
For technical details on the functional calculus for multiplication operators and holomorphic functional calculus we will use in this proof, we refer to \cite[Ch.~2]{isem21} and \cite[Ch.~2]{Ha06}. 
Let
\begin{equation*}
f(z)=\frac{\Phi^{-1}(z)}{1+z}.
\end{equation*}
It suffices to prove that $f(-A)$ is bounded, where $f(-A)$ is defined via the measurable functional calculus if we consider (i) and via the holomorphic functional calculus if we consider (ii). Indeed, we obtain from the properties of the functional calculi that
\begin{equation*}
f(-A) (1-A) \subseteq \Phi^{-1}(-A)
\end{equation*}
in the sense of inclusion of the respective graphs of the operators. If $f(-A)$ is bounded the operator on the left-hand side lies in $\Lb(D(A),X)$ and so does $\Phi^{-1}(-A)$.
We distinguish between the two possible assumptions:
\begin{enumerate}[label=(\rm \roman*)]
\item Since $f$ is a bounded function on $[0,\infty)$ we derive from the functional calculus for multiplication operators that $f(-A)$ is bounded.
\item To prove that $f(-A)$ is a bounded operator, it suffices to prove that there exist $c,\alpha>0$ such that
\begin{equation}\label{eq:H_0^infinity}
\left \lvert f(z) \right\rvert \leq c \min(\lvert z \rvert^\alpha , \lvert z \rvert^{-\alpha} ) \quad \text{ for all } z \in S_\delta,
\end{equation}
By \Cref{assump:Phi^-1_holomorph}, $\Phi^{-1}$ is holomorphic on some sector $S_\delta$ and $\lvert\Phi^{-1}(z)\rvert\leq m_{1}\Phi^{-1}(\lvert z\rvert)$ for $z\in S_{\delta}$. 
Since $\Phi\in\mathcal{P}$, we infer by \eqref{eq:ineq_for_class_P_direct} that, for $\lvert z \rvert \leq 1$,
\begin{equation*}
\frac{\Phi^{-1}(\lvert z \rvert)}{\lvert 1+z \rvert} \leq \Phi^{-1}(\lvert z \rvert) \leq \Phi^{-1}(1) \lvert z \rvert^{\nicefrac{1}{q}},
\end{equation*}
and, by \eqref{eq:ineq_for_class_P_transformed} and \eqref{eq:behavior_Phi^-1*tildePhi^-1}, that, for $\lvert z \rvert \geq 1$,
\begin{equation*}
\frac{\Phi^{-1}(\lvert z \rvert)}{\lvert 1+ z \rvert} \leq \frac{\Phi^{-1}(\lvert z \rvert)}{\lvert z \rvert} \leq \frac{2}{\widetilde{\Phi}^{-1}(\lvert z \rvert)} \leq \frac{2}{\widetilde{\Phi}^{-1}(1)} \lvert z \rvert^{-\nicefrac{1}{p'}}.  \text{\qedhere}
\end{equation*}
\end{enumerate}
\end{proof}

Recall \Cref{example:class_P} of Young functions of class $\mathcal{P}$. While Example (iii) is only useful when $A$ is a multiplication operator, Examples (ii) and (iv) yield Young functions $\Phi$ which satisfy \Cref{assump:Phi^-1_holomorph}. Example (i) tells us how to construct further examples of class $\mathcal{P}$, e.g. $\rho(t)=t^{r} + \log(t)$, $r \in [0,1]$, yields $\Phi \in \mathcal{P}$ via \eqref{eq:Phi^-1calssP} for any choice of $1<p<q<\infty$. However, in general it is not clear whether this construction leads to functions satisfying \Cref{assump:Phi^-1_holomorph} again.

%

\begin{lemma}\label{lem:measfunccalcest}
Suppose that $A$ generates a bounded analytic semigroup. If either
\begin{enumerate}[label=$\bullet$]
\item $A$ is a multiplication operator with $\sigma(-A)\subseteq[0,\infty)$ or 
\item \Cref{assump:Phi^-1_holomorph} holds and $\Phi \in \mathcal{P}$,
\end{enumerate}
then we have that
\begin{equation*}
\sup_{t > 0}~ ( \Phi^{-1}(\tfrac{1}{t}) )^{-1} ~\lVert  \Phi^{-1}(-A) T(t) \rVert < \infty.
\end{equation*}
\end{lemma}

\begin{proof}
Let $t>0$ and $f(s) \coloneqq \Phi^{-1}(s) \e^{-st}$. If $A$ is a multiplication operator, then $f(-A)=\Phi^{-1}(-A) T(t)$ holds and $\lVert f(-A) \rVert \leq \sup_{s\geq 0} f(s)$. First, note that $s \mapsto s \e^{-st}$ attains its maximum at $s=\frac{1}{t}$ and $s \mapsto \frac{\Phi^{-1}(s)}{s}$ is decreasing, since $\Phi^{-1}$ is concave. Hence, for $s \geq \frac{1}{t}$ we conclude that
\begin{equation*}
f(s) = \frac{\Phi^{-1}(s)}{s} \cdot s \e^{-st} \leq \frac{\Phi^{-1}\left(\frac{1}{t}\right)}{\frac{1}{t}} \cdot \tfrac{1}{t} \e^{-1} = f\left(\tfrac{1}{t}\right)
\end{equation*}
and $\sup_{s \geq 0} f(s)= \sup_{s \in [0,\frac{1}{t}]} f(s)$. Monotonicity of $\Phi^{-1}$ yields that $\sup_{s \in [0,\frac{1}{t}]} \Phi^{-1}(s) \\ = \Phi^{-1}(\tfrac{1}{t})$. For $s \in [0,\frac{1}{t}]$ it follows that
\begin{equation*}
\Phi^{-1}\left(\tfrac{1}{t}\right)~ \e^{-1}
\leq \sup_{s \in[0,\frac{1}{t}]} \Phi^{-1}(s)~ \e^{-st}
\leq \Phi^{-1}\left(\tfrac{1}{t}\right).
\end{equation*}
Thus, there exists $c \in [e^{-1},1]$ with $\sup_{s \geq 0}f(s)= c ~\Phi^{-1}(\tfrac{1}{t})$ and the assertion follows.

Let \Cref{assump:Phi^-1_holomorph} hold and let $\Phi$  be in $\mathcal{P}$. Let $\omega, \delta, m_1$ be as in \Cref{assump:Phi^-1_holomorph}, choose $\delta' \in (\omega,\delta)$ and take $\Gamma=\partial S_{\delta'}$ orientated positively. Then, 
\begin{align*}
\Vert \Phi^{-1}(-A)T(t) \rVert
&\leq \frac{m_1}{2 \pi} \int_\Gamma \Phi^{-1}(\lvert z \rvert) \e^{-\real z t} \lVert R(z,-A) \rVert \d z\\
&\leq \frac{m_1 M_{\delta'}}{2 \pi} \int_\Gamma \frac{\Phi^{-1}(\lvert z \rvert)}{\lvert z \rvert} \e^{- \real z t} \d z\\
&= \frac{m_1 M_{\delta'}}{ \pi} \int_0^\infty \frac{\Phi^{-1}(r)}{r} \e^{- r \cos(\delta') t} \d r\\
&\overset{s=rt}{=} \frac{m_1 M_{\delta'}}{\pi} \int_0^\infty \frac{\Phi^{-1}(\frac{s}{t})}{s} \e^{-s \cos(\delta')} \d s\\
& \leq \Phi^{-1}\left(\tfrac{1}{t}\right) ~\frac{m_1 M_{\delta'}}{\pi} \int_0^\infty \max(s^{\nicefrac{1}{p} -1},s^{\nicefrac{1}{q} -1})~ \e^{-s \cos(\delta')} \d s
\end{align*}
where we used \eqref{eq:ineq_for_class_P_direct} in the last step. Since the last integral converges, the proof is complete.
\end{proof}

\begin{remark}\label{rem:class_P_for_singulaity_in_0}
We want to point out that $\Phi \in \mathcal{P}$ is only needed to guarantee $\Phi^{-1}(-A) \in \Lb(D(A),X)$ and to deal with the singularity of the integrand at $0$. If we consider the integral over $(\varepsilon,\infty)$ with $\varepsilon \in (0,1]$ we derive the estimate
\begin{equation*}
\int_\varepsilon^\infty \frac{\Phi^{-1}(\frac{s}{t})}{s} \e^{-s \cos(\delta')} \d s \leq  \frac{\Phi^{-1}\left( \frac{1}{t} \right)}{\varepsilon} \int_\varepsilon^\infty \e^{-s \cos(\delta')} \d s,
\end{equation*}
since $s \mapsto \frac{\Phi^{-1}(\frac{s}{t})}{s}$ is decreasing and $\Phi^{-1}$ is increasing.
\end{remark}

\begin{lemma}\label{lem:auxlemma1}
Suppose that $A$ generates a bounded analytic semigroup. If either
\begin{enumerate}[label=$\bullet$]
\item $A$ is a multiplication operator with $\sigma(-A)\subseteq[0,\infty)$ or 
\item \Cref{assump:Phi^-1_holomorph} holds and $\Phi \in \mathcal{P}$
\end{enumerate} 
and if $\Phi^{-1}(-A)$ is $\Lp_\Phi$-admissible, then for every $\tau>0$ there exists $c_{\tau}>0$ such that
\begin{equation}\label{lem:auxlemma1eq}
\lVert t \Phi^{-1}(\tfrac{1}{t}) T(t)Ax \rVert_{\Lp_\Phi(0,\tau;X)} \leq c_{\tau} \lVert x \rVert
\end{equation}
holds for all $x \in D(A)$.\\
If $\Phi^{-1}(-A)$ is infinite-time $\Lp_\Phi$-admissible, then \eqref{lem:auxlemma1eq} holds for $\tau=\infty$ and $c_\infty<\infty$.
\end{lemma}

\begin{proof}
Define $f:[0,\infty) \rightarrow [0,\infty)$ by $f(s) = \frac{s}{\Phi^{-1}(s)} \e^{- \nicefrac{st}{2}}$ for $s>0$ and $f(0)=0$. Similar to \Cref{lem:measfunccalcest} we show that $t \Phi^{-1}(\tfrac{1}{t}) f(-A)$ is uniformly bounded in $t> 0$.\\
First, suppose that $A$ is a multiplication operator. Note that $f$ attains its maximum in $[0,\frac{2}{t}]$, since
\begin{equation*}
f(s) = \frac{1}{\Phi^{-1}(s)} \cdot s\e^{-\nicefrac{st}{2}} \leq f(\tfrac{2}{t})
\end{equation*}
holds for $s \geq \frac{2}{t}$. The function $s \mapsto \frac{s}{\Phi^{-1}(s)}$ is increasing. Hence, for $s \in [0,\frac{2}{t}]$ we obtain that $f(s) \leq \frac{2}{t\Phi^{-1}(\frac{2}{t})} \leq \frac{2}{t\Phi^{-1}(\frac{1}{t})} $, where we used the monotonicity of $\Phi^{-1}$ in the last inequality. This shows that
\[ \lVert f(-A) \rVert \leq \sup_{s \geq 0} \lvert f(s) \rvert \leq \frac{2}{t \Phi^{-1}(\tfrac{1}{t})} \]
and hence the claimed uniform boundedness follows. \\
Second, consider \Cref{assump:Phi^-1_holomorph} and $\Phi \in \mathcal{P}$. Let $\delta$ and $m_0$ be given as in \Cref{assump:Phi^-1_holomorph}. Choose $\delta' \in (\omega_A,\delta)$ and let $\Gamma=S_{\delta'}$ be orientated positively. Then,
\begin{align*}
\lVert f(-A) \rVert 
&\leq \frac{1}{2 \pi m_0} \int_\Gamma \frac{\lvert z \rvert}{\Phi^{-1}(\lvert z \rvert)} \e^{-\real z \, \nicefrac{t}{2}} \lVert R(z,-A) \rVert \d z\\
&\leq \frac{M_{\delta'}}{2 \pi m_0} \int_\Gamma \frac{\e^{-\real z \, \nicefrac{t}{2}}}{\Phi^{-1}(\lvert z \rvert)}  \d z\\
&=\frac{M_{\delta'}}{\pi m_0} \int_0^\infty \frac{\e^{- \cos(\delta') r \, \nicefrac{t}{2}}}{\Phi^{-1}(r)}  \d r\\
&\overset{s=rt}{=} \frac{M_{\delta'}}{\pi m_0} \int_0^\infty \frac{\e^{- \cos(\delta') \, \nicefrac{s}{2}}}{t \Phi^{-1}(\frac{s}{t})}  \d s\\
&\leq \frac{1}{t \Phi^{-1}(\frac{1}{t})} \frac{M_{\delta'}}{\pi m_0} \int_0^\infty \max(s^{-\nicefrac{1}{p}},s^{-\nicefrac{1}{q}})~ \e^{- \cos(\delta') \, \nicefrac{s}{2}} \d s
\end{align*}
by \eqref{eq:ineq_for_class_P_transformed}. Since the last integral converges, $t \Phi^{-1}(\tfrac{1}{t}) f(-A)$ is uniformly bounded in $t$.\\
Next, we use the following identity for $x\in D(A)$,
\begin{equation*}
t \Phi^{-1}(\tfrac{1}{t}) T(t)Ax= -t \Phi^{-1}(\tfrac{1}{t}) f(-A) ~ \Phi^{-1}(-A)T(\tfrac{t}{2}) x.
\end{equation*}
While the first part is uniformly bounded, we can apply the admissibility estimate from \Cref{def:adm} to the second part to obtain the desired estimate.\\
Note that we can decompose the operator in the above way by the properties of the functional calculus. Indeed, $f(-A)$ is bounded and $\ran T(\tfrac{t}{2}) \subseteq D(A) \subseteq D(\Phi^{-1}(-A))$ for all $t >0$, where the first inclusion is a known fact for bounded analytic semigroups.
\end{proof}

\begin{corollary}\label{cor:tCT(t)AxinLPhi}
Suppose that $A$ generates a bounded analytic semigroup. If either
\begin{enumerate}[label=$\bullet$]
\item $A$ is a multiplication operator with $\sigma(-A)\subseteq[0,\infty)$ or 
\item \Cref{assump:Phi^-1_holomorph} holds and $\Phi \in \mathcal{P}$
\end{enumerate} 
and if $\Phi^{-1}(-A)$ is $\Lp_\Phi$-admissible and $C \in \Lb(D(A),Y)$ satisfies 
\[\sup_{t >0} ~ (\Phi^{-1}(\tfrac{1}{t}))^{-1} \lVert C (\e^{-\beta t} T(t)) \rVert < \infty\]
for some $\beta \geq 0$, then for every $\tau>0$ there exist constants $c_{\tau}, K_{\tau}>0$ such that
\begin{equation}\label{cor:tCT(t)AxinLPhieq}
\lVert t C (\e^{-\beta t}T(t)) (A-\beta)x \rVert_{\Lp_\Phi(0,\tau;X)} \leq \left(c_{\tau} + K_{\tau} \beta \right)  \lVert x \rVert
\end{equation}
holds for all $x \in D(A)$.\\
If $\Phi^{-1}(-A)$ is infinite-time $\Lp_\Phi$-admissible, then \eqref{cor:tCT(t)AxinLPhieq} holds for $\tau=\infty$ and $c_\infty<\infty$.
\end{corollary}

\begin{proof}
For $x \in D(A)$ we write
\[ tC(\e^{-\beta t} T(t))Ax= (\Phi^{-1}(\tfrac{1}{t}))^{-1} C (\e^{-\beta \, \nicefrac{t}{2}} T(\tfrac{t}{2})) ~ t \Phi^{-1}(\tfrac{1}{t}) (\e^{-\beta \, \nicefrac{t}{2}}T(\tfrac{t}{2}))(A-\beta)x.\]
Since $(\Phi^{-1}(\tfrac{1}{t}))^{-1} C (\e^{-\beta \, \nicefrac{t}{2}}T(\tfrac{t}{2}))$ is uniformly bounded by the assumptions it suffices to estimate $t \Phi^{-1}(\tfrac{1}{t}) (\e^{-\beta \, \nicefrac{t}{2}}T(\tfrac{t}{2}))(A-\beta)x$. It follows from \Cref{lem:auxlemma1} that
\[ \lVert t \Phi^{-1}(\tfrac{1}{t}) (\e^{-\beta \, \nicefrac{t}{2}} T(\tfrac{t}{2}))Ax \rVert_{\Lp_\Phi(0,\tau;X)} \leq 2 \lVert \tfrac{t}{2}  ~\Phi^{-1}(\tfrac{2}{t}) T(\tfrac{t}{2})Ax \rVert_{\Lp_\Phi(0,\tau;X)} \leq c_{\tau} \lVert x \rVert.\]
for some $c_{\tau}$ which is uniformly bounded in $\tau$ if $\Phi^{-1}(-A)$ is infinite-time $\Lp_\Phi$-admissible. Since the semigroup is bounded and $t \mapsto t \Phi^{-1}(\tfrac{1}{t})$ is bounded on $[0,\tau]$ there exists a constant $\widetilde{K}_\tau>0$ such that
\[ \lVert t \Phi^{-1}(\tfrac{1}{t}) \e^{- \beta \, \nicefrac{t}{2}} T( \tfrac{t}{2}) \beta \rVert \leq \beta \widetilde{K}_\tau .\]
A straight-forward estimate of the Orlicz norm completes the proof.
\end{proof}


We briefly introduce the weak Orlicz space $\Lp_{\Phi,\infty} = \Lp_{\Phi,\infty}(0,\infty;Y)$ which consists of all measurable functions $f:[0,\infty) \rightarrow Y$ (modulo functions that are zero almost everywhere) such that
\[\lVert f \rVert_{\Lp_{\Phi,\infty}} \coloneqq \sup_{t\geq 0} (\Phi^{-1}(\tfrac{1}{t}))^{-1} ~ f^*(t) < \infty, \]
where $f^*$ denotes the decreasing rearrangement of $f$,
\begin{align*}
f^*(t) 
&\coloneqq \inf \{ s \geq 0 \mid \lambda(\{ \omega \in [0,\infty) \mid \lVert f(\omega) \rVert >s \}) <t\}\\
&=\inf \{ s \geq 0 \mid \lambda([\lVert f \rVert >s]) <t\}
\end{align*}
with the abbreviation $[g>s] \coloneqq \{ \omega \in[0,\infty) \mid g(\omega)>s\}$ for any function $g$ on $[0,\infty)$. As usually, we just write $\Lp_{\Phi,\infty}(0,\infty)$ if $Y=\C$. The reader is referred to \cite{liu2013weak,montgomery2011orlicz} for more on weak Orlicz spaces and Orlicz-Lorentz spaces.

\begin{theorem}\label{thm:equivalence_resolventest&semigroupest}
If $A$ generates a bounded analytic semigroup $(T(t))_{t \geq 0}$ and if $\Phi \in \mathcal{P}$, then the following statements are equivalent for $C \in \Lb(D(A),Y)$
\begin{enumerate}[label= \rm (\roman*)]
\item The infinite-time $\Phi$-Weiss condition holds, i.e. \eqref{eq:Phi_Weiss_condition} holds with $\alpha=0$, 
\item $\sup_{t>0} ( \Phi^{-1}(\tfrac{1}{t}) )^{-1} ~ \lVert CT(t) x \rVert \leq M \lVert x \rVert$ for some $M>0$ and all $x \in X$.
\item $C$ is infinite-time $\Lp_{\Phi,\infty}$-admissible.
\end{enumerate}
\end{theorem}

\Cref{thm:equivalence_resolventest&semigroupest} generalizes \cite[Thm.~2.3]{Haak2012} and \cite[Lem~2.3]{hamid2010direct}. In \cite{Haak2012} the above lemma was proved for $\Phi(t)=t^2$ and our proof of ``(ii) $\Rightarrow$ (iii) $\Rightarrow$ (i)'' is based on this source. In \cite{hamid2010direct} the equivalence of (i) and (ii) was shown for $\Phi(t) = t^p$ and our proof  of (i) $\Rightarrow$ (ii) relies on the ideas of \cite{hamid2010direct}.


\begin{proof}[Proof of \Cref{thm:equivalence_resolventest&semigroupest}]
To prove (i) $\Rightarrow$ (ii) let $\delta<\frac{\pi}{2}$ be larger than $\omega$, the type of sectoriality of $A$ and choose $M>0$ such that $\lVert \lambda R(\lambda,A )\rVert \leq M$ on $S_\delta$. We assume $\Gamma = \partial S_\delta'$ to be orientated positively for $\delta' \in (\delta,\omega)$. For $z \in \Gamma$ the resolvent equation yields
\begin{equation*}
\lVert  CR(-z,A) \rVert = \lVert CR(z,A) (I + 2z R(-z,A) \rVert \leq (1+2M) \lVert CR(z,A) \rVert.
\end{equation*}
 Then, for any $t>0$ and $x \in X$, we estimate
 \begin{align*}
 \lVert CT(t)x \rVert
 &\leq \frac{1}{2 \pi} \int_\Gamma \e^{-\real z t} \lVert CR(z,-A)x \rVert \d z \\
 &\leq \frac{1+2M}{2 \pi} \int_\Gamma \e^{-\real z t} \lVert CR(z,A)x \rVert \d z \\
 &\leq M' \lVert x \rVert \int_\Gamma \frac{\e^{-\real z t}}{\widetilde{\Phi}^{-1}(\real z)} \d z 
 \end{align*}
 where $M'>0$ is a suitable constant according to our assumptions. For $z \in \Gamma$ we can write $\real z =\lvert z \rvert \cos(\delta')$ and hence,
 \begin{align*}
 \lVert CT(t)x \rVert 
 &\leq 2M' \lVert x \rVert \int_0^\infty \frac{\e^{-\cos(\delta') t r}}{\widetilde{\Phi}^{-1}(r)} \d r \\
 &\overset{s=rt}{=} 2M' \lVert x \rVert \int_0^\infty \frac{\e^{-\cos(\delta') s}}{t \widetilde{\Phi}^{-1}(\frac{s}{t})} \d s\\
 &\leq \frac{2M'}{t \widetilde{\Phi}^{-1}(\frac{1}{t})} \lVert x \rVert \int_0^\infty \max(s^{-\nicefrac{1}{q'}},s^{-\nicefrac{1}{p'}}) ~ \e^{-\cos(\delta')s} \d s \\
 &\leq 4M' \Phi^{-1}(\tfrac{1}{t}) \int_0^\infty \max(s^{-\nicefrac{1}{q'}},s^{-\nicefrac{1}{p'}}) ~ \e^{-\cos(\delta')s} \d s 
 \end{align*}
 where we applied \eqref{eq:ineq_for_class_P_transformed} and \eqref{eq:behavior_Phi^-1*tildePhi^-1}. Since the integral is finite, we are done.

Next, we prove (ii) $\Rightarrow$ (iii). Let $M$ be given as in (ii). For $x \in X$ we have that $\lambda([ \lVert CT(\cdot)x \rvert >s]) \leq \lambda([\Phi^{-1}(\tfrac{1}{\cdot}) M \lVert x \rVert])= (\Phi(\frac{s}{M \lVert x \rVert}))^{-1}$ and hence
\begin{align*}
\lVert CT(\cdot) x \rVert_{\Lp_{\Phi,\infty}(0,\infty;Y)}
&= \sup_{t > 0} (\Phi^{-1}(\tfrac{1}{t}))^{-1} (CT(\cdot)x)^*(t)\\
&\leq \sup_{t > 0} (\Phi^{-1}(\tfrac{1}{t}))^{-1} \inf\{s \geq 0 \mid (\Phi(\tfrac{s}{M \lVert x \rVert}))^{-1}<t \}\\
&= \sup_{t > 0} (\Phi^{-1}(\tfrac{1}{t}))^{-1} \Phi^{-1}(\tfrac{1}{t}) M \lVert x \rVert \\
&= M \lVert x \rVert.
\end{align*}
This shows that $C$ is infinite-time $\Lp_{\Phi,\infty}$-admissible.

To complete the proof we show (iii) $\Rightarrow$ (i). For $z \in \C_0$ be given. The function $g:[0,\infty) \rightarrow [0,\infty)$, $g(t) = \e^{-\real z t}$ is decreasing and hence $g=g^*$. Let $x \in X$ and set $f(t)=CT(t)x$. The Hardy--Littlewood inequality yields for every $z \in \C_0$ that
\begin{align*}
\lVert CR(z,A)x \rVert
&= \int_0^\infty g(t) f(t) \d t\\
&\leq \int_0^\infty t f^*(t) \frac{1}{t} g^*(t) \d t\\
&\leq 2 \int_0^\infty \frac{1}{\Phi^{-1}(\frac{1}{t})} f^*(t) \frac{g^*(t)}{t \widetilde{\Phi}^{-1}(\frac{1}{t})} \d t\\
&\leq 2 \lVert f \rVert_{\Lp_{\Phi,\infty}} \int_0^\infty \frac{\e^{- \real z t}}{t \widetilde{\Phi}^{-1}(\frac{1}{t})} \d t\\
&\overset{s=\real z t}{=} 2 \lVert f \rVert_{\Lp_{\Phi,\infty}} \int_0^\infty \frac{\e^{- s}}{s \widetilde{\Phi}^{-1}(\frac{\real z}{s})} \d s\\
&\leq \int_0^\infty \frac{\e^{-s}}{s \min(s^{-\nicefrac{1}{q'}},s^{-\nicefrac{1}{p'}}) \widetilde{\Phi}^{-1}(\real z) } \d s \\
&\leq  \frac{K \lVert f \rVert_{\Lp_{\Phi,\infty}}}{\widetilde{\Phi}^{-1}(\real z)},
\end{align*}
for some $K>0$, where we applied \eqref{eq:behavior_Phi^-1*tildePhi^-1} and \eqref{eq:ineq_for_class_P_transformed}. By assumption, $\lVert f \rVert_{\Lp_{\Phi,\infty}} \leq c_\infty \lVert x \rVert$ with admissibility constant $c_\infty< \infty$. Hence, (i) follows and the proof is complete.
\end{proof}

As a direct consequence of a rescaling argument for the semigroup we can formulate the finite-time version of \Cref{thm:equivalence_resolventest&semigroupest}.
\begin{corollary}\label{cor:equivalence_resolventest&semigroupest_infinite_time}
If $A$ generates a bounded analytic semigroup and if $\Phi \in \mathcal{P}$, then the following statements are equivalent for $C \in \Lb(D(A),Y)$
\begin{enumerate}[label= \rm (\roman*)]
\item The $\Phi$-Weiss condition \eqref{eq:Phi_Weiss_condition} holds for some $\alpha>\omega(A)$, 
\item $\sup_{t>0} ( \Phi^{-1}(\tfrac{1}{t}) )^{-1} ~ \lVert C(\e^{-\beta t}T(t)) x \rVert \leq M \lVert x \rVert$ for some $\beta>\omega(A)$, $M>0$ and all $x \in X$.
\item $C$ is $\Lp_{\Phi,\infty}$-admissible.
\end{enumerate}
In (i) and (ii) the parameters $\alpha$ and $\beta$ can be chosen the same if they are non-negative.
\end{corollary}

The approach to the proof of d's result from \cite{hamid2010direct} uses the boundedness of the integral operator $L$, defined for $\tau>0$ by
\begin{equation}\label{eq:integralop.L}
(Lf)(t) \coloneqq \int_t^{\tau} \frac{f(s)}{s} \d s \quad 0\leq t\leq \tau
\end{equation}
on $\Lp^p(0,\tau;Y)$ with operator norm bounded by $p$, see \cite[Prop.~2.2]{hamid2010direct}. As a direct consequence of the interpolation result from \cite[Thm.~5.1]{karlovich2001interpolation} we have the following lemma.

\begin{lemma}\label{lem:integraloponLPhi}
If $\Phi \in \mathcal{P}$ and $L$ is given by \eqref{eq:integralop.L} for some $\tau>0$ then $L$ is a bounded operator on $\Lp_\Phi(0,\tau;Y)$ with operator norm independent of $\tau>0$.
\end{lemma}

We put everything together to get our main theorem:

\begin{theorem}\label{thm:Phi_Weiss_conjecture}
Suppose that $A$ generates a bounded analytic semigroup. If $\Phi \in \mathcal{P}$ and either $A$ is a multiplication operator with $\sigma(-A)\subseteq[0,\infty)$ or \Cref{assump:Phi^-1_holomorph} holds. Then the following are equivalent
\begin{enumerate}[label=\rm (\roman*)]
\item  $\Phi^{-1}(-A)$ is (infinite-time) $\Lp_\Phi$-admissible,
\item It holds that 
\begin{equation*}
C \text{ is (infinite-time) } \Lp_\Phi \text{-admissible} \Leftrightarrow \left\{
\begin{aligned}
& C \text{ satisfies the (infinite-time) } \\
 & \Phi\text{-Weiss condition \eqref{eq:Phi_Weiss_condition}}
 \end{aligned}\right.
\end{equation*}
\end{enumerate}
\end{theorem}

\begin{proof}
Since $A$ generates a bounded semigroup we have that $\omega(A) \leq 0$.\\
First, assume (ii). From \Cref{lem:measfunccalcest} and \Cref{thm:equivalence_resolventest&semigroupest} we deduce the infinite-time $\Phi$-Weiss condition (and hence the finite-time $\Phi$-Weiss condition) holds for $C=\Phi^{-1}(-A)$. Thus, $\Phi^{-1}(-A)$ is (infinite-time) $\Lp_\Phi$-admissible by (ii).\\
Second, assume (i). If $C$ is (infinite-time) $\Lp_\Phi$-admissible, then the (infinite-time) $\Phi$-Weiss property \eqref{eq:Phi_Weiss_condition} follows. This was \Cref{lem:adm_implies_Weiss_condition}.\\
It is left to prove that the (infinite-time) $\Phi$-Weiss property for $C$ implies (infinite-time) $\Lp_\Phi$-admissibility of $C$. First consider the finite-time case.
Let 
\[\sup_{z \in \C_\alpha} \widetilde{\Phi}^{-1}(\real z) ~ \lVert CR(z,A) \rVert < \infty\]
for some $\alpha > \omega(A)$.
\Cref{cor:equivalence_resolventest&semigroupest_infinite_time} implies for $\beta>\max(\alpha,0)$ that
\[M \coloneqq \sup_{t>0} (\Phi^{-1}(\tfrac{1}{t}))^{-1} \lVert C (\e^{-\beta t}T(t))\rVert < \infty\]
 and \Cref{cor:tCT(t)AxinLPhi}  implies that $f(t)=t C(\e^{-\beta t}T(t)) (A-\beta)x$ lies in $\Lp_\Phi(0,\tau;Y)$ for every $\tau \in (0,\infty)$. For $x \in D(A)$ and $t \in (0,\tau)$ we have that
\begin{align*}
C(\e^{-\beta t}T(t))x&=
C(\e^{-\beta \tau}T(\tau))x - \int_t^{\tau} C(\e^{-\beta s}T(s))(A-\beta)x \d s \\
&= C(\e^{-\beta \tau}T(\tau))x - (Lf)(t),
\end{align*}
where $L$ is the integral operator given by \eqref{eq:integralop.L}, which is bounded on $\Lp_\Phi(0,\tau;Y)$ by \Cref{lem:integraloponLPhi} since $\Phi\in \mathcal{P}$. We obtain that
\begin{align*}
\lVert C(\e^{-\beta t}T(t))x \rVert_{\Lp_\Phi(0,\tau;Y)} 
&\leq \lVert C(\e^{-\beta \tau}T(\tau))x \rVert_{\Lp_\Phi(0,\tau;Y)} + \lVert Lf \rVert_{\Lp_\Phi(0,\tau;Y)} \\
&\leq (\Phi^{-1}(\tfrac{1}{\tau}))^{-1} ~ \lVert C(\e^{-\beta \tau}T(\tau))x \rVert_Y + \lVert L \rVert \lVert f \rVert_{\Lp_\Phi(0,\tau;Y)} \\
&\leq [M+ \lVert L \rVert (c_{\tau} + \beta K_{\tau})] \lVert x \rVert
\end{align*}
where $c_{\tau}$ and $K_{\tau}$ are the constants from \Cref{cor:tCT(t)AxinLPhi} and $\lVert L \rVert$ denotes the operator norm of $L$ on $\Lp_\Phi(0,\tau;Y)$.
This shows that $C$ is $\Lp_\Phi$-admissible. The infinite-time case is even simpler. Assume that the infinite-time $\Phi$-Weiss condition holds. \Cref{thm:equivalence_resolventest&semigroupest} implies that
\[M \coloneqq \sup_{t>0} (\Phi^{-1}(\tfrac{1}{t}))^{-1} \lVert C T(t)\rVert < \infty\]
and as before
\[\lVert CT(t)x \rVert_{\Lp_\Phi(0,\tau;Y)}  \leq  (M+ \lVert L \rVert c_{\tau}) \lVert x \rVert. \]
Since $\lVert L \rVert$ and $c_{\tau}$ are uniformly bounded in $\tau>0$ (c.f.~\Cref{cor:tCT(t)AxinLPhi}) we obtain that $C$ is infinite-time $\Lp_\Phi$-admissible.
\end{proof}

On $X=\ell^r$, $r \in [1,\infty)$, there is a sufficient condition on $\Phi$ for infinite-time $\Lp_\Phi$-admissibility of $\Phi^{-1}(-A)$ when dealing with a multiplication operator $A$ given by
\begin{equation}\label{eq:diagA}
Ae_n=\lambda_n e_n,
\end{equation}
 where  $(e_n)_{n \in \N}$ is the standard basis on $\ell^r$ and $(\lambda_n)_n$ is assumed to be a sequence of non-positive numbers, i.e. $\lambda_n \leq 0$ for all $n \in \N$. 
 The  domain of $A$ is given by
\[D(A) = \left\{x=(x_n)_{n \in \N} \in \ell^r \: \middle\vert \: \sum_{n=1}^\infty \lvert \lambda_n x_n \rvert^r < \infty \right\}\]
and it is well-known that $A$ is the generator of a bounded analytic semigroup $(T(t))_{t\geq 0}$ given by
\[T(t)e_n = \e^{\lambda_n t} e_n, \qquad n \in\mathbb{N}. \]
Clearly, for any Young function $\Phi$ the functional calculus for multiplication operators yields that $\Phi^{-1}(-A)$ is given by
\[\Phi^{-1}(-A) e_n = \Phi^{-1}(-\lambda_n)e_n, \qquad n \in\mathbb{N}.\]

\begin{proposition}\label{prop:admissibility_on_l^r}
Consider the operator $A$ on $\ell^r$ given by \eqref{eq:diagA}. If $\Phi$ and $t\mapsto \Phi(t^{\nicefrac{1}{r}})$ are Young functions, then $\Phi^{-1}(-A)$ is infinite-time $\Lp_\Phi$-admissible.
\end{proposition}

\begin{proof}
Similar to the $\Lp^p$ case we obtain for $x=(x_n)_n \in D(A)$ that
\begin{align*}
\lVert \Phi^{-1}(-A) T(\cdot)x \rVert_{\Lp_\Phi(0,\infty;\ell^r)}
&= \left\lVert\left( \sum_{\substack{n=1 \\ \lambda_n \neq 0}}^\infty \lvert \Phi^{-1}(-\lambda_n) \e^{\lambda_n \cdot}  x_n \rvert^r  \right)^{\nicefrac{1}{r}}\right\rVert_{\Lp_\Phi(0,\infty)}\\
&\leq 2^{\nicefrac{1}{r}} \left(\sum_{\substack{n=1 \\ \lambda_n \neq 0}}^\infty \lVert \Phi^{-1}(-\lambda_n) \e^{\lambda_n \cdot} x_n \rVert_{\Lp_\Phi(0,\infty)}^r \right)^{\nicefrac{1}{r}}\\
&\leq 2^{\nicefrac{1}{r}} \lVert x \rVert_{\ell^r}
\end{align*}
holds, where we applied the generalized Minkowski inequality (\Cref{apx:Minkowski}) and \Cref{apx:Orlicz_norm_exp_function}. This proves that $\Phi^{-1}(-A)$ is infinite-time $\Lp_\Phi$-admissible.
\end{proof}

\begin{remark}
We would like to point out that the theory developed in this section is also applicable to selfadjoint operators $A$ on Hilbert spaces. Indeed, by the spectral theorem (see~\cite[Thm.~D.5.1]{Ha06}) $A$ is unitary equivalent to a multiplication operator and admissibility of $C$ for $A$ is preserved under unitary transformations of $C$ and $A$. 
\end{remark}

\section{Duality of control and observation operators}\label{sec:Duality}

The previous results on Orlicz admissibility for observation operators can easily be transfered to control operators via duality. We extend known duality results for $\Lp^p$, see e.g. \cite{Weiss89i}, to Orlicz spaces.

Let $X,U,Y$ be Banach spaces such that $A$ is the generator of a $C_0$-semigroup $(T(t))_{t \geq 0}$ on $X$ and the dual semigroup $(T'(t))_{t \geq 0}$ is also strongly continuous. This is in particular the case if $X$ is reflexive. By $X_1$ we denote $D(A)$ equipped with the graph norm of $A$ and $X_{-1}$ is the completion of $X$ with respect to the norm defined by
\begin{equation*}
\lVert x \rVert_{X_{-1}} = \lVert (\beta -A)^{-1} x \rVert_X, \quad x \in X
\end{equation*}
for some $\beta \in \rho(A)$. Different choices of $\beta \in \rho(A)$ leads to equivalent norms on $X_{-1}$. The $C_0$-semigroup $(T(t))_{t \geq 0}$ has a unique extension to a $C_0$ semigroup $(T_{-1}(t))_{t \geq 0}$ on $X_{-1}$ whose generator is an extension of $A$. The same construction can be done for $A'$ which leads to spaces $X_1^d$ and $X_{-1}^d$. We have continuous and dense embeddings
\begin{equation*}
X_1 \hookrightarrow X \hookrightarrow X_{-1} \quad \text{ and } \quad X_1^d \hookrightarrow X' \hookrightarrow X_{-1}^d,
\end{equation*}
and dual pairings
\begin{equation*}
\langle y_1 , x_1 \rangle_{X_{-1}^d, X_1} \quad \text{ and } \quad \langle y_2 , x_2 \rangle_{X_1^d , X_{-1}}
\end{equation*} 
which are nothing else as the standard dual pairing on $X'$ and $X$ if $y_1 \in X'$ and $x_2 \in X$, see e.g.~\cite{TucWei09}. The spaces $X_{-1}$ and $X_{-1}^d$ are the dual spaces of $X_1^d$ and $X_1$ with respect to the pivot spaces $X$ and $X'$ respectively (\cite[Ch.~6]{Weiss89i}). This concept gives rise to dual operators for $B \in \Lb(U,X_{-1})$ and $C \in \Lb(X_1,Y)$ in the following sense $B' \in \Lb(X_1^d,U')$ and $C' \in \Lb(Y',X_{-1}^d)$. Thus a control operator $B \in \Lb(U,X_{-1})$ of the control system
\begin{equation*}
\left\{ \begin{array}{ll}
\dot{x}(t) &= Ax(t) + Bu(t), \ t \geq 0,\\
x(0) &= x_0,\\
\end{array}\right.
\end{equation*}
can be viewed as a observation operator of the dual observation system in the sense of \eqref{eq:observationsystem} in $X'$ with $A'$ instead of $A$. In the same manner $C'$ can be seen as a control operator of a system in $X'$.

\begin{definition}
The control operator $B \in \Lb(U,X_{-1})$ is called {\em $\Lp_\Phi$-admissible} (for $A$ or $(T(t))_{t \geq 0}$) if for some (and hence for all) $t > 0$ there exists a minimal constant $b_\tau > 0$ such that
\begin{equation*}
\left\lVert \int_0^\tau T_{-1}(t-s) Bu(s) \d s \right\rVert_X \leq b_\tau \lVert u \rVert_{\Lp_\Phi(0,\tau;U)}.
\end{equation*}
If, additionally, $b_\infty\coloneqq \sup_{\tau \geq 0} b_\tau < \infty$, then $B$ is called {\em infinite-time $\Lp_\Phi$-admissible}.
\end{definition}

We give the following duality result, which extends Weiss' result for $\Lp^p$  \cite{Weiss89i}.

\begin{theorem}
\begin{enumerate}[label=\arabic*.]
\item The observation operator $C \in \Lb(X_1,Y)$ is (infinite-time) $\Lp_\Phi$-admissible if and only if $C' \in \Lb(Y',X_{-1}^d)$ is a (infinite-time) $\Lp_{\widetilde{\Phi}}$-admissible control operator. Moreover, the admissibility constants coincide.
\item If the control operator $B \in \Lb(U,X_{-1})$ is (infinite-time) $\Lp_\Phi$-admissible then $B' \in \Lb(X_1^d,U')$ is a (infinite-time) $\Lp_{\widetilde{\Phi}}$-admissible observation operator. Moreover, denoting the admissibility constants of $B$ and $B'$ by $b_\tau$ and $c_\tau$, we have that $c_\tau \leq b_\tau$. Equivalence holds if $U$ is reflexive, in this case the admissibility constants coincide.
\end{enumerate}
\end{theorem}

\begin{proof}
By \Cref{apx:equivalent norms_Orlicz_Luxemburg} the proof is analogous to the one of \cite[Thm.~6.3]{Weiss89i}. \end{proof}

\section{Appendix: Further results on Orlicz spaces}

We give some technical results on Orlicz spaces, some of which we could not locate in the literature, and hence may be of interest in their own right. This includes \Cref{apx:dual_Orlicz_space}, \Cref{apx:rightshift_EPhi}, \Cref{apx:rightshift_LPhi} and \Cref{apx:Minkowski}. The functions $\Phi$ and $\widetilde{\Phi}$ will always be complementary Young functions and $(\Omega,\mathcal{F},\mu)$ will denote a measure space. \medskip

We start off with a technical result on the class $\mathcal{P}$, c.f.~\Cref{sec:Orlicz_spaces} and the Orlicz norm of parameter depending exponential functions which is of interest when working with the Laplace transform representation of the resolvent of an operator.

\begin{lemma}\label{apx:P_functions_are_Young_functions}
Let $f:(0,\infty) \rightarrow (0,\infty)$ be given by 
\begin{equation*}
f(t) = t^{\nicefrac{1}{p}} \rho(t^{\nicefrac{1}{q}-\nicefrac{1}{p}}) \quad \text{ with } 1 < p < q < \infty
\end{equation*}
for $t>0$, where $\rho:(0,\infty) \rightarrow (0,\infty)$ is a continuous concave function such that $\rho(st) \leq \max(1,s) \rho(t)$ for all $s,t>0$. Then $f$ is strictly increasing and hence invertible on $(0,\infty)$. Its inverse $f^{-1}$ is a Young function.
\end{lemma}

\begin{proof}
From $\rho(st)\leq \max(1,s) \rho(t)$ we deduce that $\rho$ is increasing. The concavity of $\rho$ implies that $s \mapsto \frac{\rho(s)}{s}$ is decreasing on $(0,\infty)$ and since $\frac{1}{q} - \frac{1}{p}<0$, $s \mapsto \frac{\rho(st)^{\nicefrac{1}{q}-\nicefrac{1}{p}}}{s^{\nicefrac{1}{q}-\nicefrac{1}{p}}}$ is increasing for every $t>0$. For $s \in (0,1)$
\begin{align*}
f(st)
&= s^{\nicefrac{1}{p}} t^{\nicefrac{1}{p}} \rho((st)^{\nicefrac{1}{q}-\nicefrac{1}{p}})\\
&\geq s^{\nicefrac{1}{p}} t^{\nicefrac{1}{p}} \rho(t^{\nicefrac{1}{q}-\nicefrac{1}{p}})
= s^{\nicefrac{1}{p}} f(t)
\end{align*}
and for $s \in [1,\infty)$
\begin{align*}
f(st)
&=s^{\nicefrac{1}{q}} t^{\nicefrac{1}{p}} \frac{\rho((st)^{\nicefrac{1}{q}-\nicefrac{1}{p}})}{s^{\nicefrac{1}{q}-\nicefrac{1}{p}}}\\
&\geq s^{\nicefrac{1}{q}} t^{\nicefrac{1}{p}}\rho(t^{\nicefrac{1}{q}-\nicefrac{1}{p}})
=s^{\nicefrac{1}{q}} f(t).
\end{align*}
These inequalities and the properties of $\rho$ imply
\begin{equation}\label{eq:behavior_of_class_P_functions}
\min(s^{\nicefrac{1}{p}},s^{\nicefrac{1}{q}}) f(t) \leq f(st) \leq \max(s^{\nicefrac{1}{p}},s^{\nicefrac{1}{q}}) f(t)
\end{equation}
for $s,t>0$. We conclude that $f$ is strictly increasing on $(0,\infty)$ and $f((0,\infty) ) = (0,\infty)$. Hence, $f$ possesses an inverse $f^{-1}:(0,\infty) \rightarrow (0,\infty)$, which is again continuous and strictly increasing. For the convexity of $f^{-1}$ we refer to \cite{karlovich2001interpolation}. It follows from \eqref{eq:behavior_of_class_P_functions} for $t=1$ that $\frac{f(s)}{s} \rightarrow \infty$ as $s \to 0$ and $\frac{f(s)}{s} \to 0$ as $s \to \infty$  holds, hence $f^{-1}$ is a Young function.
\end{proof}

\begin{lemma}\label{apx:Orlicz_norm_exp_function}
It holds for every $s>0$ that
\[  \widetilde{\Phi}^{-1}(s) \leq \left(\lVert \e^{-s \cdot} \rVert_{\Lp_{\widetilde{\Phi}}(0,\infty)}\right)^{-1} \]
and if $\widetilde{\Phi} \in \Delta_2^{\text{\tiny global}}$, then there exists a constant $c>0$ such that
\begin{equation}\label{eq:exporlicznormequivalent}
c \left(\lVert \e^{-s \cdot} \rVert_{\Lp_{\widetilde{\Phi}}(0,\infty)}\right)^{-1}   \leq \widetilde{\Phi}^{-1}(s) \leq \left(\lVert \e^{-s \cdot} \rVert_{\Lp_{\widetilde{\Phi}}(0,\infty)}\right)^{-1}.
\end{equation}
\end{lemma}

\begin{proof}
The convexity of $\widetilde{\Phi}$ yields for $k \geq \left( \widetilde{\Phi}^{-1}(s) \right)^{-1}$
\begin{equation*}
\int_0^\infty \widetilde{\Phi}\left( \frac{\e^{-st}}{k} \right) \d t \leq \widetilde{\Phi}\left(\frac{1}{k}\right) \int_0^\infty \e^{-st} \d t  \leq 1
\end{equation*}
and hence, $\lVert \e^{-s \cdot} \rVert_{\Lp_{\widetilde{\Phi}}(0,\infty)} \leq \left( \widetilde{\Phi}^{-1}(s) \right)^{-1}$. For the second part assume $\widetilde{\Phi} \in \Delta_2^{\text{\tiny global}}$, i.e. there is $K>1$ such that $\widetilde{\Phi}(e x) \leq K \widetilde{\Phi}(x)$ for all $x>0$. By monotonicity 
\[ \widetilde{\Phi}(\e^r x) \leq \widetilde{\Phi}(\e^{\lceil r \rceil}  x) \leq  K^{\lceil r \rceil} \widetilde{\Phi}(x) \leq K^{r+1} \widetilde{\Phi}(x)\]
 follows for all $r>0$ and taking $x=e^{-r} \widetilde{\Phi}^{-1}(s)$ leads to 
 \[ K^{-(r+1)} s \leq \widetilde{\Phi}(\e^{-r} \widetilde{\Phi}^{-1}(s)).\]
 Let $c=\min \{1, \frac{1}{K\log(K)}\} \in (0,1]$. Convexity of $\widetilde{\Phi}$ yields
\begin{align*}
\int_0^\infty \widetilde{\Phi}\left( \frac{\e^{-st}\widetilde{\Phi}^{-1}(s)}{c} \right) \d t
&\geq \frac{1}{c} \int_{0}^{\infty} \widetilde{\Phi}\left( \e^{-st}\widetilde{\Phi}^{-1}(s) \right) \d t \\
&\geq \frac{1}{c} \int_0^\infty K^{- (st+1)} s \d t = \frac{1}{c \cdot K \log(K)} \geq 1.
\end{align*}
By the definition of the Luxemburg norm, we infer that
\[ c \left( \widetilde{\Phi}^{-1}(s) \right)^{-1} \leq \lVert \e^{-s \cdot} \rVert_{\Lp_{\widetilde{\Phi}}(0,\infty)},\]
 which completes the proof.
 \end{proof}
 
 \begin{proof}[Proof of \Cref{example:class_P} (iv)]
Let $\rho(t) = \log(1+t)$. It is well known that $\rho$ is concave and that it holomorphic on any sector $S_\delta$. We first check that $\rho(st) \leq \max(1,s) \rho(t)$ for $s,t>0$ holds. For $s \leq 1$ the monotonicity of $\rho$ implies $\rho(st) \leq \rho(t)$. For $s > 1$, $\rho(st) \leq s\rho(t)$ is equivalent to $\log(1+st) \leq \log( (1+t)^s)$. The letter holds by Bernoulli's inequality. Hence, $\Phi$ is of class $\mathcal{P}$.\\
For the last part, it suffices to prove $\lvert \rho(z) \rvert \sim \rho(\lvert z \rvert)$ on $S_\delta$ for $\delta \leq \frac{\pi}{3}$ or equivalently $2 \cos(\delta)\geq 1$. Let $z=r\e^{\iu \theta} \in S_\delta$, i.e. $r>0$ and $\theta \in (-\delta,\delta)$. Note that $1+z \in S_\delta$ and $\lvert 1+z \rvert^2 = 1 + 2 \cos(\theta) r + r^2$ holds. We infer that
\begin{align*}
\lvert \log(1+z) \rvert^2
&= \left\lvert \log(\lvert 1+z \rvert) + \iu \arg(1+z) \right\rvert^2\\
&=\left( \frac{1}{2} \log\left( \sqrt{1+ 2 \cos(\theta) r + r^2}\right) \right)^2 + \left( \arg(1+z) \right)^2\\
&\geq \frac{1}{16} \left( \ln(1+2 \cos(\delta) r) \right)^2\\
&\geq \frac{1}{16} \left( \ln(1+r) \right)^2\\
&=\frac{1}{16} \left( \ln(1+\lvert z \rvert) \right )^2.
\end{align*}
We proved $\frac{1}{4} \rho(\lvert z \rvert ) \leq \lvert \rho(z )\rvert$.
Similar, we estimate
\begin{align*}
\lvert \log(1+z) \rvert^2
&= \left( \frac{1}{2} \log( \underset{\leq (1+r)^2}{\underbrace{1+ 2 \cos(\theta) r + r^2}}) \right)^2 + \left( \arg(1+z) \right)^2 \\
&\leq \log(1+r)^2 + \left( \arg(1+z) \right)^2\\
&=\rho(\lvert z \rvert)^2 + \left( \arg(1+z) \right)^2.
\end{align*}
Thus,  to derive $\lvert \rho(z ) \rvert \leq m \rho(\lvert z \rvert) $ for some positive constant $m$ independent of $z$ we can show that 
\begin{equation*}
\frac{\lvert \arg(1+z) \rvert }{\rho(\lvert z \rvert)}= \frac{\lvert \arg(1+z) \rvert }{\log(1+ \lvert z \rvert)}
\end{equation*}
is bounded in $z \in S_\delta$. Since $\arg(1+z)$ is continuous in $ z \in \overline{S_\delta} \setminus \{0\}$ the boundedness follows on compact subsets of $\overline{S_\delta} \setminus \{0\}$. Moreover, $\lvert \arg(1+z)\rvert \leq \frac{\pi}{3}$ implies the boundedness for large values of $\lvert z \rvert$. It is left to show that $\frac{\lvert \arg(1+z)\rvert}{\log(1+r)}$ is bounded for small values of $\lvert z \rvert$. To this end we use $\lvert 1+z \rvert \sin(\arg(1+ z)) = \imag(1 + z)=\imag(z) = \lvert z \rvert \sin(\arg(z))$ on $S_\delta$ and that $\lvert \frac{\omega}{\sin(\omega)}\rvert \leq K$ for some $K>0$ and $\omega \in (-\delta,\delta)$. With this at hand, we estimate for $z=r \e^{\iu \theta}$,
\begin{align*}
\frac{\lvert \arg(1+z) \rvert }{\log(1+\lvert z \rvert)}
&\leq K \frac{\lvert \sin(\arg(1+r\e^{\iu \theta}))\rvert}{\log(1+r)}\\
&= K \frac{r \lvert \sin(\theta) \rvert}{\lvert 1 + r \e^{\iu \theta} \rvert ~ \log(1+r)}\\
&\leq K \frac{r}{\log(1+r)}\\
&\leq \widetilde{K}
\end{align*}
for some $\widetilde{K}>0$ and small values of $r= \lvert z \rvert$.
\end{proof}

We continue with some more general results on Orlicz spaces and the standard norms on them.

\begin{proposition}\label{apx:equivalent norms_Orlicz_Luxemburg}
For $f \in \Lp_\Phi(\Omega;Y)$  and $g \in \Lp_\Phi(\Omega;Y')$ we have that
\begin{align*}
\begin{split}
& \sup \set{ \int_\Omega \left\lvert \langle f, \widetilde{h} \rangle_{Y,Y'}\right\rvert \d \mu}{\lVert \widetilde{h} \rVert_{\Lp_{\widetilde{\Phi}}(\Omega;Y')} \leq 1}\\
 & = \sup \set{ \int_\Omega  \lVert f \rVert_Y \lvert h \rvert \d \mu }{\lVert h \rVert_{\Lp_{\widetilde{\Phi}}(\Omega)} \leq 1}
 \end{split}
\end{align*}
and
\begin{align*}
\begin{split}
& \sup \set{ \int_\Omega \left\lvert \langle \widetilde{h}, g \rangle_{Y,Y'} \right\rvert \d \mu}{\lVert \widetilde{h} \rVert_{\Lp_{\widetilde{\Phi}}(\Omega;Y)} \leq 1}\\
 & = \sup \set{ \int_\Omega  \lVert g \rVert_{Y'} \lvert h \rvert \d \mu }{\lVert h \rVert_{\Lp_{\widetilde{\Phi}}(\Omega)} \leq 1}.
 \end{split}
\end{align*}
In particular, it holds that $\interleave f \interleave_{\Lp_\Phi(\Omega;Y)} = \interleave \, \lVert f (\cdot) \rVert_{Y} \, \interleave_{\Lp_\Phi(\Omega)}$ and thus,
\begin{equation*}
\lVert f \rVert_{\Lp_\Phi(\Omega;Y)} \leq \interleave f \interleave_{\Lp_\Phi(\Omega;Y)} \leq 2 \lVert f \rVert_{\Lp_\Phi(\Omega;Y)}.
\end{equation*}
\end{proposition}

\begin{proof}
The last inequalities in a consequence of the scalar case if we have proved the rest. Hence we only have to proof equality of the supremum terms.
Note that ``$\leq$'' is clear. We have to prove the other estimate.
First consider $f \in \Lp_\Phi(\Omega;Y)$ and let $\varepsilon \in (0,1)$. From \cite[Lem.~3.4.21]{schnackers2014.Diss} with $B(x,y)= \langle x,y \rangle_{Y,Y'}$ we deduce the existence of a function (Bochner-) measurable $y: \Omega \rightarrow Y'$ with $\lVert y(\cdot) \rVert_{Y'} =1$ and
\begin{equation*}
(1-\varepsilon) \lVert y \rVert_{Y'} \leq \lvert \langle f, y \rangle_{Y,Y'} \rvert
\end{equation*} 
$\mu$-almost everywhere.
 Let $h \in \Lp_{\widetilde{\Phi}}(\Omega)$ with $\lVert h \rVert_{\Lp_{\widetilde{\Phi}}(\Omega)} \leq 1$ and set $\widetilde{h}=y h$. It follows that
\begin{align*}
& (1-\varepsilon) \int_\Omega \lVert f \rVert_Y \lvert h \rvert \d \mu\\
& \leq \sup \set{\int_\Omega \left\lvert \langle f, \widetilde{h} \rangle_{Y,Y'} \right\rvert \d \mu}{\lVert \widetilde{h} \rVert_{\Lp_{\widetilde{\Phi}}(\Omega;Y')} \leq 1}.
\end{align*}
Thus, since $\varepsilon$ and $h$ was arbitrary the assertion for $f$ follows. The proof for $g$ follows similarly.\\
By the definition of the Orlicz norm we obtain  $\interleave f \interleave_{\Lp_\Phi(\Omega;Y)} = \interleave \, \lVert f (\cdot) \rVert_{Y} \, \interleave_{\Lp_\Phi(\Omega)}$ and thus, we deduce from the scalar valued case $Y=\R$, see e.g.~\cite{Kufner}, that
\begin{equation*}
\lVert f \rVert_{\Lp_\Phi(\Omega;Y)} \leq \interleave f \interleave_{\Lp_\Phi(\Omega;Y)} \leq 2 \lVert f \rVert_{\Lp_\Phi(\Omega;Y)}.
\end{equation*}
\end{proof}

\begin{proposition}\label{apx:dual_Orlicz_space}
If $(\Omega,\mathcal{F},\mu)$ is a finite measure space then the dual space of $\E_\Phi(\Omega;Y)$ is (topologically) isomorphic to $\Lp_{\widetilde{\Phi}}(\Omega;Y')$ if and only if $Y'$ possesses the Radon-Nikodym property.
\end{proposition}

\begin{proof}
\Cref{apx:equivalent norms_Orlicz_Luxemburg} implies that the mapping $v \mapsto (u \mapsto \int_\Omega \langle u , v \rangle_{Y,Y'} \d \mu)$ from $\Lp_{\widetilde{\Phi}}(\Omega;Y')$ to $(\E_\Phi(\Omega;Y))'$ is an isomorphism onto its range. The equivalence of surjectivity of this mapping to the Radon-Nikodym property of $Y'$ can be proved analogously to the $\Lp^p$ case, see. e.g.~\cite[Ch.~IV.1]{DiestelUhl}.
\end{proof}

\begin{proposition}\label{apx:rightshift_EPhi}
The right shift semigroup $(S(t))_{t \geq 0}$ on $\E_\Phi(a,b;Y)$, $- \infty\leq a<b \leq \infty$,
\begin{equation*}
(S(t)f)(r) = \left\{ \begin{array}{ll}
f(r - t), & \text{ for } r \in (a+t,b),\\
0, & \text{ else, }\\
\end{array}\right.
\end{equation*}
is strongly continuous. Its generator is given by $D= - \frac{\d}{\d r}$ with domain $\W^1\E_\Phi(a,b;Y)$ if $a=-\infty$ and $\{ f \in \W^1\E_\Phi(a,b;Y) \mid f(a)=0 \}$ if $a>-\infty$, where $W^1\E_\Phi(a,b;Y)$ is the Orlicz-Sobolev spaces consisting of a weakly differentiable functions in $\E_\Phi(a,b;Y)$ whose weak derivative lies in $\E_\Phi(a,b;Y)$.
\end{proposition}

\begin{proof}
Since $C_0([a,b];Y)$, the space of compactly supported continuous functions is dense in $\E_\Phi(a,b;Y)$, see e.g.~\cite[Thm.~8.21]{Adams} (for $Y=\R$), the proof can be done analogously to $\Lp^p(a,b;\R)$, see~\cite[Ch.~1,~Example~5.4,~Chap.~2,~Sect.~2.10~\&~2.11]{engelnagel}. It is straight forward to lift these results to the vector valued case.
\end{proof}

Note that $f(a)=0$ in the domain of $D$ makes sense since we have continuous embeddings $\W^1\E_\Phi(a,b;Y) \hookrightarrow \W^{1,1}(a,b;Y) \hookrightarrow C([a,b];Y)$.

\begin{proposition}\label{apx:rightshift_LPhi}
The right shift semigroup $(S(t))_{t \geq 0}$ on $\Lp_\Phi(a,b;Y)$, $- \infty < a<b < \infty$, is strongly continuous if and only if $\Phi \in \Delta_2^{\infty}$. The generator is  given as in \Cref{apx:rightshift_EPhi}.
\end{proposition}

\begin{proof}
For simplicity consider $(a,b)=(0,1)$. 
If $\Phi \in \Delta_2^{\infty}$ then $\Lp_\Phi(0,1;Y) = \E_\Phi(0,1;Y)$ and the result follows from \Cref{apx:rightshift_EPhi}\\
Suppose that $\Phi \notin \Delta_2^{\infty}$. Without loss of generality we can assume $Y=\R$. We will construct a function $v \in \Lp_\Phi(0,1)$ such that $\lVert S(t) v - v \rVert_{\Lp_\Phi(0,1)}\geq 1$. Since $\Phi \notin \Delta_2^{\infty}$ there exists a sequence $(t_n)_{n \geq 1}$, $t_n \geq n$ such that $\Phi(2 t_n) \geq n \Phi(t_n)$ and $\Phi(t_n) > 1$ for all $n \geq 1$. Choose $n_0 \in \N$ such that $\sum_{n=n_0}^\infty \frac{1}{n^2} < 1$ and define a family of disjoint subintervals $(I_k)_{k \geq 1}$  of $(0,1)$ by
\begin{equation*}
I_k=\left(1- \sum_{n=n_0+k}^\infty \frac{1}{n^2} - \frac{1}{\Phi(t_k) (n_0+k-1)^2} , 1- \sum_{n=n_0+k}^\infty \frac{1}{n^2} \right).
\end{equation*}
Let $u= \sum_{k=1}^\infty t_k \mathds{1}_{I_k}$. From
\begin{equation*}
\int_0^1 \Phi(u(r)) \d r = \sum_{k=1}^\infty \Phi(t_k) \frac{1}{\Phi(t_k) (n_0+k-1)^2} = \sum_{k=n_0}^\infty \frac{1}{n^2} <1
\end{equation*}
we obtain $u \in \Lp_\Phi(0,1)$. We also have that
\begin{equation*}
\int_0^1 \Phi(2u(r)) \d r = \sum_{k=1}^\infty \Phi(2t_k) \frac{1}{\Phi(t_k) (n_0+k-1)^2} \geq \sum_{k=1}^\infty \frac{k}{(n_0+k-1)^2} = \infty.
\end{equation*}
Define $v=4u$ and note that $(S(t)v)(\cdot)$ is a bounded function for every $t>0$. The convexity of $\Phi$ implies that
\begin{align*}
\infty = \int_0^1 \Phi(2u(r)) \d r \leq \frac{1}{2} \int_0^1 \Phi(\lvert (S(t)v)(r) - v(r) \rvert ) \d r + \int_0^1 \Phi( (S(t)v)(r)) \d r
\end{align*}
and hence that $\int_0^1 \Phi(\lvert (S(t)v)(r) - v(r) \rvert ) \d r= \infty$. Thus, $\lVert S(t) v - v \rVert_{\Lp_\Phi(0,1)}\geq 1$ for every $t>0$ which means that $(S(t))_{t \geq 0}$ is not strongly continuous.
\end{proof}

\begin{remark}
It is obvious that the implication ``$\Phi \in \Delta_2^{\text{\tiny global}}$ $\Rightarrow$ the right shift semigroup is strongly continuous on $ \Lp_\Phi(a,b;Y)$'' holds even if $a=-\infty$ and $b=\infty$ (since we still have $\Lp_\Phi = \E_\Phi$). The converse holds if $b< \infty$ and $t \mapsto \frac{\Phi(t)}{\Phi(2t)}$ is bounded in $0$. This is due to the fact that $\Phi \in \Delta_2^{\text{\tiny global}}$ if and only if $\Phi \in \Delta_2^{\infty}$ and $t \mapsto \frac{\Phi(t)}{\Phi(2t)}$ is bounded in $0$ and the existence of the sequence $(t_n)_{n \geq 1}$ in the proof of \Cref{apx:rightshift_LPhi} relies on the assumption that $\Phi \notin \Delta_2^{\infty}$.
\end{remark}

Another known result for $\Lp^p$ spaces, which can be generalized to Orlicz spaces, is the Mankowski integral inequality. Although the proof is simple, the authors are not aware that such a result already exists for Orlicz spaces.

\begin{proposition}[generalized Minkowski inequality]\label{apx:Minkowski}
Let $\Phi$ be a Young function such that $\Psi(t)=\Phi(t^{\nicefrac{1}{r}})$ also defines a Young function, where $r\geq 1$. Further let $(\Omega_i , \mathcal{F}_i, \mu_i)$, $i=1,2$, be measure spaces and let $f:\Omega_1 \times \Omega_2 \rightarrow [0,\infty)$ be measurable. Then it holds that
\begin{equation*}
\left\lVert \left( \int_{\Omega_2} (f(\cdot,y))^r \d \mu_2(y) \right)^{\nicefrac{1}{r}} \right\rVert_{\Lp_\Phi(\Omega_1)} \leq 2^{\nicefrac{1}{r}} \left( \int_{\Omega_2} \lVert f(\cdot,y) \rVert_{\Lp_\Phi(\Omega_1)}^r \d \mu_2(y) \right)^{\nicefrac{1}{r}}.
\end{equation*}
The factor $2^{\nicefrac{1}{r}}$ can be omitted if we take the equivalent Orlicz norm.
\end{proposition}

\begin{proof}
First we prove the statement for $r=1$. Note that $\Psi$ is trivially a Young function in this case. Using the equivalent Orlicz norm on $\Lp_\Phi$ we obtain that
\begin{align*}
\left \lVert \int_{\Omega_2} f(\cdot,y) \d \mu_2(y) \right \rVert_{\Lp_\Phi(\Omega_1)} 
&\leq \sup_{\lVert g \rVert_{\Lp_{\widetilde{\Phi}}(\Omega_1)} \leq 1} \left\lvert \int_{\Omega_1} \int_{\Omega_2} f(x,y)~g(x) \d \mu_2(y) \d \mu_1(x) \right\rvert \\
&= \sup_{\lVert g \rVert_{\Lp_{\widetilde{\Phi}}(\Omega_1)} \leq 1} \left\lvert \int_{\Omega_2} \int_{\Omega_1} f(x,y)~g(x) \d \mu_1(x) \d \mu_2(y)\right\rvert \\
&\leq \int_{\Omega_2} \sup _{\lVert g \rVert_{\Lp_{\widetilde{\Phi}}(\Omega_1)} \leq 1} \left\lvert \int_{\Omega_1} f(x,y)~ g(x) \d \mu_1(x) \right \rvert \d \mu_2(y)\\
&\leq 2 \int_{\Omega_2} \lVert f(\cdot,y) \rVert_{\Lp_\Phi(\Omega_1)} \d \mu_2(y),
\end{align*}
where we applied the generalized H\"older inequality in the last step.\\
Now, let $r \geq 1$ be given such that $\Psi(t)=\Phi(t^{\nicefrac{1}{r}})$ defines a Young function. We deduce from the definition of the Luxemburg norm that
\begin{align*}
\left\lVert \left( \int_{\Omega_2} (f(\cdot,y))^r \d \mu_2(y) \right)^{\nicefrac{1}{r}} \right\rVert_{\Lp_\Phi(\Omega_1)}
&= \left\lVert  \int_{\Omega_2} (f(\cdot,y))^r \d \mu_2(y)  \right\rVert_{\Lp_\Psi(\Omega_1)}^{\nicefrac{1}{r}}\\
&\leq 2^{\nicefrac{1}{r}} \left(\int_{\Omega_2} \lVert (f(\cdot,y))^r \rVert_{\Lp_\Psi(\Omega_1)} \d \mu_2(y)\right)^{\nicefrac{1}{r}}\\
&=2^{\nicefrac{1}{r}} \left(\int_{\Omega_2} \lVert f(\cdot,y) \rVert_{\Lp_\Phi(\Omega_1)}^r \d \mu_2(y) \right)^{\nicefrac{1}{r}},
\end{align*}
where we applied the previous derived estimate for $r=1$ and the Young function $\Psi$.
\end{proof}


\begin{thebibliography}{10}

\bibitem{Adams}
R.~Adams.
\newblock {\em Sobolev spaces}.
\newblock Academic Press, New York-London, 1975.
\newblock {P}ure and Applied Mathematics, Vol. 65.

\bibitem{DiestelUhl}
J.~Diestel and J.~J. Uhl, Jr.
\newblock {\em Vector measures}.
\newblock Mathematical Surveys, No. 15. American Mathematical Society,
  Providence, R.I., 1977.
\newblock With a foreword by B. J. Pettis.

\bibitem{Engel.Char.Adm}
K.-J. Engel.
\newblock On the characterization of admissible control- and observation
  operators.
\newblock {\em Systems Control Lett.}, 34(4):225--227, 1998.

\bibitem{Engel.Op.Matr}
K.-J. Engel.
\newblock Spectral theory and generator property for one-sided coupled operator
  matrices.
\newblock {\em Semigroup Forum}, 58(2):267--295, 1999.

\bibitem{engelnagel}
K.-J. Engel and R.~Nagel.
\newblock {\em One-parameter semigroups for linear evolution equations}, volume
  194 of {\em Graduate Texts in Mathematics}.
\newblock Springer-Verlag, New York, 2000.

\bibitem{CallGrab96}
P.~Grabowski and F.~M. Callier.
\newblock Admissible observation operators. {S}emigroup criteria of
  admissibility.
\newblock {\em Integral Equations Operator Theory}, 25(2):182--198, 1996.

\bibitem{Haakthesis}
B.~Haak.
\newblock {\em {Kontrolltheorie in Banachr\"aumen und quadratische
  Absch\"atzungen}}.
\newblock PhD thesis, University of Karlsruhe, 2004.

\bibitem{Haak2012}
B.~H. Haak.
\newblock The {W}eiss conjecture and weak norms.
\newblock {\em J. Evol. Equ.}, 12(4):855--861, 2012.

\bibitem{HaakKunst07}
B.~H. Haak and P.~C. Kunstmann.
\newblock Weighted admissibility and wellposedness of linear systems in
  {B}anach spaces.
\newblock {\em SIAM J. Control Optim.}, 45(6):2094--2118, 2007.

\bibitem{Ha06}
M.~Haase.
\newblock {\em The functional calculus for sectorial operators}, volume 169 of
  {\em Operator Theory: Advances and Applications}.
\newblock Birkh\"{a}user Verlag, Basel, 2006.

\bibitem{isem21}
M.~Haase.
\newblock Functional calculus, 2017/18.
\newblock Lecture notes of the 21st {I}nternet {S}eminar, available at
  https://www.math.uni-kiel.de/isem21/en/course/phase1/isem21-lectures-on-functional-calculus.

\bibitem{hamid2010direct}
B.~Hamid, D.~Adberrahim, and E.-M. Omar.
\newblock A direct approach to the {W}eiss conjecture for bounded analytic
  semigroups.
\newblock {\em Czechoslovak Mathematical Journal}, 60(2):527--539, 2010.

\bibitem{JNPS16}
B.~Jacob, R.~Nabiullin, J.~Partington, and F.~Schwenninger.
\newblock {Infinite-dimensional input-to-state stability and {O}rlicz spaces}.
\newblock {\em SIAM J. Control Optim.}, 56(2):868--889, 2018.

\bibitem{JaPa01}
B.~Jacob and J.~R. Partington.
\newblock The {W}eiss conjecture on admissibility of observation operators for
  contraction semigroups.
\newblock {\em Integral Equations Operator Theory}, 40(2):231--243, 2001.

\bibitem{jpsurvey}
B.~Jacob and J.~R. Partington.
\newblock Admissibility of control and observation operators for semigroups: a
  survey.
\newblock In {\em Current Trends in Operator Theory and its Applications},
  volume 149 of {\em Oper. Theory Adv. Appl.}, pages 199--221. Birkh\"{a}user,
  Basel, 2004.

\bibitem{JPP02}
B.~Jacob, J.~R. Partington, and S.~Pott.
\newblock Admissible and weakly admissible observation operators for the right
  shift semigroup.
\newblock {\em Proc. Edinb. Math. Soc. (2)}, 45(2):353--362, 2002.

\bibitem{JaZw04}
B.~Jacob and H.~Zwart.
\newblock Counterexamples concerning observation operators for
  {$C_0$}-semigroups.
\newblock {\em SIAM J. Control Optim.}, 43(1):137--153, 2004.

\bibitem{karlovich2001interpolation}
A.~Karlovich and L.~Maligranda.
\newblock On the interpolation constant for {O}rlicz spaces.
\newblock {\em Proceedings of the American Mathematical Society},
  129(9):2727--2739, 2001.

\bibitem{KrasnRut}
M.~Krasnosel{\cprime}ski{\u \i} and Y.~Ruticki{\u \i}.
\newblock {\em Convex functions and {O}rlicz spaces}.
\newblock {Translated from the first Russian edition by Leo F. Boron}. P.
  Noordhoff Ltd., Groningen, 1961.

\bibitem{Kufner}
A.~Kufner, O.~John, and S.~Fu\v{c}\'{\i}k.
\newblock {\em Function spaces}.
\newblock Noordhoff International Publishing, Leyden; Academia, Prague, 1977.
\newblock Monographs and Textbooks on Mechanics of Solids and Fluids;
  Mechanics: Analysis.

\bibitem{LeMerdy2003}
C.~Le~Merdy.
\newblock The {W}eiss conjecture for bounded analytic semigroups.
\newblock {\em J. London Math. Soc. (2)}, 67(3):715--738, 2003.

\bibitem{liu2013weak}
P.~Liu and M.~Wang.
\newblock Weak {O}rlicz spaces: {S}ome basic properties and their applications
  to harmonic analysis.
\newblock {\em Science China Mathematics}, 56(4):789--802, 2013.

\bibitem{montgomery2011orlicz}
S.~Montgomery-Smith.
\newblock {O}rlicz-{L}orentz spaces.
\newblock {\em Mathematics publications (MU)}, 2011.

\bibitem{schnackers2014.Diss}
C.~Schnackers.
\newblock {\em {O}rlicz-{M}odulationsr\"aume}.
\newblock PhD thesis, Aachen, Techn. Hochsch., Diss., 2014, 2014.

\bibitem{TucWei09}
M.~Tucsnak and G.~Weiss.
\newblock {\em Observation and {C}ontrol for {O}perator {S}emigroups}.
\newblock Birkh{\"a}user Advanced Texts: Basler Lehrb{\"u}cher. Birkh{\"a}user
  Verlag, Basel, 2009.

\bibitem{TucWei14}
M.~Tucsnak and G.~Weiss.
\newblock Well-posed systems---the {LTI} case and beyond.
\newblock {\em Automatica J. IFAC}, 50(7):1757--1779, 2014.

\bibitem{Weiss89i}
G.~Weiss.
\newblock {Admissible observation operators for linear semigroups}.
\newblock {\em Israel J. Math.}, 65(1):17--43, 1989.

\bibitem{We90}
G.~Weiss.
\newblock Two conjectures on the admissibility of control operators.
\newblock In {\em Estimation and control of distributed parameter systems
  ({V}orau, 1990)}, volume 100 of {\em Internat. Ser. Numer. Math.}, pages
  367--378. Birkh\"{a}user, Basel, 1991.

\bibitem{Wynn11}
A.~Wynn.
\newblock Sufficient conditions for weighted admissibility of operators with
  applications to {C}arleson measures and multipliers.
\newblock {\em Q. J. Math.}, 62(3):747--770, 2011.

\end{thebibliography}

\def\cprime{$'$}

\end{document}